\documentclass{amsart}

\usepackage{amsmath}
\usepackage{amssymb}
\usepackage{cite}
\usepackage{graphicx, epsfig}
\usepackage{hyperref}
\usepackage{soul}
\usepackage{xcolor}

\usepackage[margin=3cm]{geometry}

\newtheorem{theorem}{Theorem}[section]
\theoremstyle{plain}

\newtheorem{corollary}[theorem]{Corollary}

\newtheorem{definition}[theorem]{Definition}

\newtheorem{proposition}[theorem]{Proposition}
\newtheorem{remark}[theorem]{Remark}

\numberwithin{equation}{section}

\usepackage{hyperref}
\hypersetup{
    colorlinks   = true,
    urlcolor     = blue,
    citecolor    = red,
    bookmarksopen=true
}

\title{Lane--Emden Systems with Singular Nonlinearities for the Fully Nonlinear Elliptic Operator}
\author{Karan Rathore}
\address[Karan Rathore]{Visvesvaraya National Institute Of Technology, India-440010}
\email{ds23mth004@students.vnit.ac.in}
\author{Mohan Mallick}
\address[Mohan Mallick]{Visvesvaraya National Institute Of Technology, India-440010} 
\email{mohan.math09@gmail.com, mohanmallick@mth.vnit.ac.in}
\author{Ram Baran Verma}
\address[Ram Baran Verma]{Department of Mathematics, SRM University Amaravati, Andhra Pradesh-522502, India} \email{rambaran.v@srmap.edu.in}


\begin{document}
\begin{abstract}
\noindent Consider
\[
\begin{cases}
F(D^2 u,Du,u,x) = u^{-p}v^{-q},~\text{in}~\Omega\\
F(D^2 v,Dv,v,x)=u^{-r}v^{-s},~~\text{in}~~\Omega\\
u,v>0~~\text{in}~~\Omega\\
u=v=0~\quad~\text{on}~~\partial\Omega,
\end{cases}
\]
where $\Omega$ is an open connected subset of $\mathbb{R}^{N}$ and $p,s$ are two non-negative and $q,r$ are positive real numbers. This article discuses the conditions in terms of the relations among $p,q,r$ and $s$ which lead to existence, uniqueness and non-existence of positive solutions to the system. Furthermore, we also have studied some regularity properties of solution of the system. These results are inspired by the study of Lane-Emden system of equations as in \cite{busca2002liouville,ghergu2010lane}.
\end{abstract}
\maketitle
\vspace{0cm}
\section{Introduction}
This article is concerned with the following system of equations
\begin{equation}\label{main}
\left\{
\begin{aligned}
F(D^2u,Du,u,x) &= u^{-p}v^{-q} && \text{in } \Omega,\\
F(D^2v,Dv,v,x) &= u^{-r}v^{-s} && \text{in } \Omega,\\
u = v &= 0 && \text{on } \partial\Omega,
\end{aligned}
\right.
\end{equation}
where $\Omega \subset \mathbb{R}^N$ is a bounded domain with $C^{2}$ boundary, $p,s \geq 0$, and $q,r>0$. We discus various criterion in terms of relations among the exponents $p,q,r$ and $s,$ which lead to existence, uniqueness and nonexistence of  solutions to system \eqref{main}. In order discuss the context of our work among the existence work, let us consider the following problem:
\begin{equation}\label{supp}
\left\{
\begin{aligned}
a_{ij}D_{ij}u&= u^{\tilde{p}} && \text{in } \Omega,\\
u&=0~~&&\text{on}~\partial\Omega,
\end{aligned}
\right.
\end{equation}
where $1\leq \tilde{p}.$ This problem in the literature is called Lane–Emden equation. There are various mathematical models which lead to   
Lane–Emden system of equations see\cite{chandrasekhar1957introduction,quittner2007superlinear,vazquez2006porous}. The problem \eqref{supp} has non-variational structure and the usual approach to study the existence of positive solution via topological degree. One of the crucial steps in this approach is to obtain uniform a prior estimate for positive solutions. B. Gidas and J. Spruck develop a technique to obtain the uniform bound by blow up technique \cite{Gidas01011981,gidas1981global}. It requires the Liouville type theorem which gives the "non-existence" of non-trivial solution of $\Delta u= u^{\tilde{p}}$ in $\mathbb{R}^{N}.$ To derive analogous estimates in nonconvex smooth domains, it is necessary to identify the range of $\tilde{p}$ for which $\Delta u= u^{\tilde{p}}$ admits no nonnegative solutions in the half-space $\mathbb{R}_+^N$, except $u\equiv0$. In the context of fully nonlinear elliptic equations, this type problem with $\tilde{p}>0$ has been studied in \cite{A,AB2,AF1}. This result has been extended for various class of operators see for example\cite{quaas2006existence,AB2,de2014priori,A,AF1,de2014liouville}. For more results in this directions see section 4\cite{tyagi2016survey}. Similar results for the system of uniformly elliptic equations has been obtained in \cite{busca2002liouville,quaas2009existence,clkment2014positive,serrin1996non}.\\
The system of equation \eqref{main} with $F=-\Delta,$ $p,s<0$ and $q,r\leq0$ has received considerable attention and investigation does not concern only the existence of solution but also many qualitative properties of solution by using moving plane method and Pohozaev-type identities etc,\cite{de2013liouville,de2014liouville,busca2002liouville,clement2000existence,zou2002priori,serrin1996non,naito2006existence,souplet2009proof,reichel2000non,quittner2004priori}. The system of equations \eqref{main} with Laplace operator and $p,q<0<r,s$ corresponds to the singular Gierer–Meinhardt system arising in molecular biology\cite{ghergu2009steady,ghergu2007class,ghergu2008singular}. The analysis of the system of equation in this case has been studied in \cite{hernandez2008positive}, where authors use monotone iteration to establish the existence of positive solution.\\
Later on the Problem \eqref{main} in the context of Laplace operator has been studied in \cite{ghergu2010lane}. In his study author establish many criterion in terms of the relation among the exponents on the nonlinear term which leads to existence, uniqueness and non-existence of solution. In accomplishing these results author frequently constructed appropriate sub and supersolutions of auxiliary equations. These sub and supersolutions are in terms of eigen functions of associated problems. In this work, we extend the results in \cite{ghergu2010lane} to the fully nonlinear setting, where the Laplacian is replaced by a general operators $F$ which involves dependent variable and its gradient. Moreover, these operators are of non-divergence form in nature unlike the Laplace operator which has both non variational as well variational. Main results of this article heavily relies on the boundary behaviour of solutions to the system \eqref{main}. These results have been deduced as a consequence of boundary H\"{o}pf Lemma and the Proposition \eqref{prop4}. This proposition deals with the boundary behaviour of solution of \eqref{eq2} which is proved by constructing appropriate sub and supersolution. These sub and supersolutions are obtained as a function(which is a solution of ordinary differential equation \eqref{H}) of eigen function in \eqref{eq:eigen}. Finally we use the properties of solution of \eqref{H} and the fact that eigenfunction is comparable with distance function \eqref{ineigen}. At this point we would like to emphasize that the boundary behaviour of solution of \eqref{eq2} with the nonlinearity behaving like $\delta(x)^{q}u^{-p}$ in place of $\delta(x)^{-q}u^{-p}$ has been established in \cite{felmer2012existence} where authors constructed the barrier function as a appropriate function of distance function directly. For more general results in this direction we refer to \cite{mallick2025regularity}. On one hand these boundary behaviour of solution combined with a suitable improper integral gives the criterion of non-existence of solution to \eqref{main} see Theorem \eqref{nonexis1}. On the other hand under suitable criterion, by using the results of \eqref{prop4}, we define suitable cone in the function space and have used the Schauder fixed point theorem in this cone prove the existence of solution. Finally, in the same section we have used these results to prove the regularity of solution of system.\\

 This article is organized as follows. In Section~2, we introduce the assumptions and structural conditions imposed on the operator $F$. Section~3 is devoted to the construction of barrier functions and the derivation of boundary estimates. The main non-existence, existence, regularity and uniqueness results are established in Section~4.

\section{Preliminaries}

Given two positive numbers $\lambda \leq \Lambda$, we define the Pucci extremal operators as follows:
\begin{equation*}
\mathcal{P}_{\lambda,\Lambda}^+(M)
= \sup_{A \in [\lambda,\Lambda]} \bigl[-\operatorname{Tr}(AM)\bigr]
\quad \text{and} \quad
\mathcal{P}_{\lambda,\Lambda}^-(M)
= \inf_{A \in [\lambda,\Lambda]} \bigl[-\operatorname{Tr}(AM)\bigr],
\end{equation*}
where $[\lambda,\Lambda]$ denotes the set of real symmetric matrices whose eigenvalues lie in the interval $[\lambda,\Lambda]$.
These operators provide canonical examples of fully nonlinear uniformly elliptic operators.

This article deals with a more general class of fully nonlinear elliptic operators.
Let
\[
F : S(N) \times \mathbb{R}^{N} \times \mathbb{R} \times \Omega \to \mathbb{R}
\]
satisfy the following structural conditions:
\begin{equation}\label{sc}
\left\{
\begin{aligned}
(H1)\;& F(tM,tp,tr,x)=tF(M,p,r,x), \quad t\ge 0.\\
(H2)\;& |F(M,p,r,x)-F(M,p,r,y)|
\le C|x-y|^{\alpha}(1+|M|+|p|+|r|).\\
(H3)\;& F \text{ is convex in } (M,p,r).\\
(H4)\;& \mathcal{P}_{\lambda,\Lambda}^-(M-N)-\Gamma|p-q|-\gamma(r-s)^+ \\
&\le F(M,p,r,x)-F(N,q,s,x) \\
&\le \mathcal{P}_{\lambda,\Lambda}^+(M-N)+\Gamma|p-q|+\gamma(s-r)^- .
\end{aligned}
\right.
\end{equation}

This article concerns a decoupled system involving fully nonlinear elliptic operators.
Such operators are of non-divergence type in nature.
Therefore, the appropriate notion of solution for studying these problems is that of viscosity solutions.

The viscosity solution of the system
\begin{equation}\label{problem1}
\left\{
\begin{aligned}
F(D^2u_{1},Du_1,u_1,x)&=f_{1}(u_{1},u_{2}) \quad &&\text{in } \Omega,\\
F(D^2u_{2},Du_2,u_2,x)&=f_{2}(u_{1},u_{2}) \quad &&\text{in } \Omega,\\
u_{1}&=u_{2}=0 \quad &&\text{on } \partial\Omega,
\end{aligned}
\right.
\end{equation}
is defined as follows.
\begin{definition}[see \cite{MHP,ishii1991viscosity}]\label{viscosity}
A vector-valued function $\vec{u}=(u_{1},u_{2}) \in
C(\overline{\Omega}) \times C(\overline{\Omega})$ is called a viscosity
subsolution (resp., supersolution) of \eqref{problem1} if, for any
$\phi \in C^{2}(\Omega)$ such that $u_i-\phi$ (for some $i=1,2$) attains
a local maximum (resp., local minimum) at a point $x_{0} \in \Omega$,
the following holds:
\begin{equation*}
F(D^2\phi(x_{0}),D\phi(x_{0}),\phi(x_{0}),x_{0})
\le f_{i}(u_{1}(x_{0}),u_{2}(x_{0}))
\quad
\bigl(\text{resp., } \ge f_{i}(u_{1}(x_{0}),u_{2}(x_{0}))\bigr).
\end{equation*}
The function $\vec{u}$ is called a viscosity solution of \eqref{problem1}
if it is both a viscosity subsolution and a viscosity supersolution.
\end{definition}

Unless otherwise stated, by a solution we always mean a viscosity solution.
We recall the following theorem from \cite{armstrong2009principal}.
\subsection{Eigenvalue}

\begin{theorem}[Theorem 2.2, \cite{armstrong2009principal}]\label{distance}
Let $F$ satisfies the above assumptions. Then there exist $\varphi^{+}_{1},\varphi^{-}_{1} \in C^{1,\alpha}(\Omega)$
such that 
\begin{equation}\label{eq:eigen}
\left\{
\begin{aligned}
F\big(D^{2}\varphi^{+}_{1}, D\varphi^{+}_{1}, \varphi^{+}_{1}, x\big) 
   =& \mu^{+}_{1}(F,\Omega)\,\varphi^{+}_{1}~ \text{in} ~\Omega,\\
F\big(D^{2}\varphi^{-}_{1}, D\varphi^{-}_{1}, \varphi^{-}_{1}, x\big) 
   =& \mu^{-}_{1}(F,\Omega)\,\varphi^{-}_{1}~\text{in }~\Omega,\\
\varphi^{+}_{1} =& \varphi^{-}_{1} = 0~\text{on}~\partial\Omega.
\end{aligned}
\right.
\end{equation}
and $\varphi^{+}_{1}> 0\quad\text{and}~\quad \varphi^{-}_{1} < 0\quad \text{in}~\Omega.$ Moreover, the eigenvalue $\mu^{+}_{1}(F,\Omega)$ (resp. $\mu^{-}_{1}(F,\Omega)$) is unique in the sense that if $\mu$ is an eigenvalue of $F$ in $\Omega$ with a corresponding nonnegative (resp. nonpositive) eigenfunction, then 
\[ \mu = \mu^{+}_{1}(F,\Omega) \quad 
\big(\text{resp. } \mu = \mu^{-}_{1}(F,\Omega)\big).
\]
Additionally, it is also simple in the sense that if $\varphi \in C(\overline{\Omega})$ 
is a solution of \eqref{eq:eigen} with $\varphi$ in place of $\varphi^{+}_{1}$ (resp. $\varphi^{-}_{1}$), 
then $\varphi$ is a constant multiple of $\varphi^{+}_{1}$ (resp. $\varphi^{-}_{1}$).
\end{theorem}
\begin{remark}
\begin{enumerate}
\item
As a consequence of the H\"{o}pf lemma, there exists a neighborhood
$N$ of $\partial\Omega$ such that
\begin{equation}\label{rem1}
|D\varphi^{+}_{1}(x)| \ge c
\quad \text{for all } x \in N \cap \Omega.
\end{equation}
For the proof of the H\"{o}pf lemma, see Appendix~A of \cite{armstrong2009principal}.

\item
Following the ideas of Theorem~4.9 in \cite{AB}, we obtain that the
eigenfunction $\varphi^{+}_{1}$ is Lipschitz continuous. Moreover, by
using the H\"{o}pf lemma, we can show that there exists a constant
$C>0$ such that
\begin{equation}\label{ineigen}
C\,\delta(x) \le \varphi^{+}_{1}(x)
\le \frac{1}{C}\,\delta(x)
\quad \text{for all } x \in \Omega.
\end{equation}
\end{enumerate}
\end{remark}
We state the following result, which will be used repeatedly throughout the paper.
Let $F$, $F_{1}$, and $F_{2}$ be nonlinear operators satisfying
\begin{equation}\label{eq:2.28}
F(M+N,\, p+q,\, z+w,\, x)
\leq F_{1}(M,p,z,x)+F_{2}(N,q,w,x)
\end{equation}
for all $M,N \in S^n$, $p,q \in \mathbb{R}^n$, $z,w \in \mathbb{R}$, and
$x \in \Omega$.

\begin{proposition}[Lemma~2.3 in \cite{armstrong2009principal}]\label{compute}
Let $F_{1}$, $F_{2}$, and $F$ be nonlinear operators satisfying
\textup{(H1)}, \textup{(H2)}, and \eqref{eq:2.28}.
Let $u \in C(\overline{\Omega})$ be a subsolution of
\begin{equation}\label{eq:2.29}
F_{1}(D^2u, Du, u, x)=f \quad \text{in } \Omega,
\end{equation}
and let $v \in C(\overline{\Omega})$ be a subsolution of
\begin{equation}\label{eq:2.30}
F_{2}(D^2v, Dv, v, x)=g \quad \text{in } \Omega,
\end{equation}
for some $f,g \in C(\Omega)$.
Then the function $w := u+v$ is a subsolution of
\begin{equation}\label{eq:2.31}
F(D^2w, Dw, w, x)=f+g \quad \text{in } \Omega.
\end{equation}
Likewise, if we reverse the inequality in \eqref{eq:2.28} and assume that
$u$ and $v$ are supersolutions of \eqref{eq:2.29} and \eqref{eq:2.30},
respectively, then $w$ is a supersolution of \eqref{eq:2.31}.
\end{proposition}

We will use the following results concerning the comparison principle.
\begin{proposition}\label{prop1}
Let $p \ge 0$ and let $\phi : \Omega \to (0,\infty)$ be a continuous function.
If $\underline{u}$ is a subsolution and $\overline{u}$ is a supersolution of
\begin{equation*}
\left\{
\begin{aligned}
F(D^2u,Du,u,x) &= \phi(x)u^{-p} && \text{in } \Omega,\\
u &= 0 && \text{on } \partial\Omega,
\end{aligned}
\right.
\end{equation*}
then
\[
\underline{u} \le \overline{u} \quad \text{in } \Omega.
\]
\end{proposition}

\begin{proof}
We observe that assumption \textup{(H4)} implies that $F$ is a proper elliptic
operator, that is, increasing. Therefore, it is easy to note that $F(M, p,., x)-\phi u^{-p}$ is also non-decreasing in $u$. So by Theorem 3.3 in \cite{crandall1992user}, comparison principle holds.

\end{proof}
In the next section, we derive lower bounds for subsolutions and supersolutions near the boundary using the above tools, and, when necessary, construct logarithmic-type barrier functions.
\section{Boundary Behavior and Singular Estimates}
In this section, we derive lower bounds for subsolutions and supersolutions of singular elliptic equations of the form \eqref{main}. These estimates are expressed in terms of the distance to the boundary, $\delta(x)=\operatorname{dist}(x,\partial\Omega)$, and capture the precise singular behavior near $\partial\Omega$.
\begin{proposition}\label{prop3}
Let $(u,v)$ be a solution of the system \eqref{main}. Then there exists a constant
$c>0$ such that
\begin{equation}\label{exeq}
u(x)\ge c\,\delta(x)
\quad \text{and} \quad
v(x)\ge c\,\delta(x)
\quad \text{in } \Omega.
\end{equation}
\end{proposition}

\begin{proof}
Observe that $F$ is convex in its first three arguments.
Therefore, we may follow the ideas of \cite{quaas2008principal}.
Since $(u,v)$ is a solution of the system \eqref{main}, we have
\begin{equation*}
\left\{
\begin{aligned}
F(D^2u(x),Du(x),u(x),x)
&= u^{-p}v^{-q} > 0
\quad \text{in } \Omega,\\
u &= 0
\quad \text{on } \partial\Omega.
\end{aligned}
\right.
\end{equation*}
By following the same argument as in Theorem~1.4 of
\cite{quaas2008principal}, we obtain
\[
u \equiv t\,\varphi^{+}_{1}
\quad \text{for some } t>0.
\]
Consequently, the first inequality in \eqref{exeq} follows from
\eqref{ineigen}. The second inequality can be obtained in a similar way.
\end{proof}
\begin{theorem}\label{th1}
Let $p \ge 0$, $\operatorname{diam}(\Omega) < A$, and let
$k : (0,A) \to (0,\infty)$ be a decreasing function such that
\[
\int_0^A t\,k(t)\,dt = \infty.
\]
Then the inequality
\begin{equation}\label{exstar}
\left\{
\begin{aligned}
F(D^2u,Du,u,x) &\ge k(\delta(x))u^{-p}
\quad \text{in } \Omega,\\
u &> 0 \quad \text{in } \Omega,\\
u &= 0 \quad \text{on } \partial\Omega,
\end{aligned}
\right.
\end{equation}
admits no solutions
\[
u \in C^2(\Omega)\cap C(\overline{\Omega}).
\]
\end{theorem}
\begin{proof}
Suppose, by contradiction, that there exists a solution $u_{0}$ of
\eqref{exstar}. Then, in view of the structural condition \textup{(H4)},
we obtain
\[
\mathcal{P}_{\lambda,\Lambda}^{+}(D^{2}u_{0})
+ \Gamma |Du_{0}|
+ \gamma (-u_{0})^{-}
\ge k(\delta(x))u_{0}^{-p}
\quad \text{in } \Omega.
\]
In particular, by taking $\Lambda=\lambda=1$ and $\Gamma=\gamma=0$, we see
that $u_{0}$ satisfies
\begin{equation}\label{exlap1}
-\Delta u_{0} \ge k(\delta(x))u_{0}^{-p}
\quad \text{in } \Omega.
\end{equation}
This contradicts Theorem~2.4 in \cite{ghergu2010lane}.
\end{proof}
Notice that for $l \ge 1$, the integral
\[
\int_{0}^{\infty} t^{-l}\,dt
\]
is infinite. Thus, if we choose $k(\delta(x))=\delta(x)^{-q}$ with $q \ge 2$
in the above theorem, we obtain the following corollary.

\begin{corollary}\label{coro}
Let $p \ge 0$ and $q \ge 2$. Then there exists no function
$u \in C^{2}(\Omega) \cap C(\overline{\Omega})$ such that
\begin{equation*}
\left\{
\begin{aligned}
F(D^2u,Du,u,x) &\ge \delta(x)^{-q}u^{-p}
&& \text{in } \Omega,\\
u &> 0
&& \text{in } \Omega,\\
u &= 0
&& \text{on } \partial\Omega.
\end{aligned}
\right.
\end{equation*}
\end{corollary}

\medskip

In view of Corollary~\ref{coro}, it is natural to consider the following
equation:
\begin{equation}\label{eq2m}
\left\{
\begin{aligned}
F(D^2u,Du,u,x) &= \delta(x)^{-q}u^{-p}
&& \text{in } \Omega,\\
u &= 0
&& \text{on } \partial\Omega,
\end{aligned}
\right.
\end{equation}
where $p \ge 0$ and $0 < q < 2$.

Our next result, Proposition~\ref{prop4}, is an analogue of
Proposition~2.6 in \cite{ghergu2010lane}. These results describe the
boundary behavior of positive solutions of \eqref{eq2m}.
We also recall that the boundary behavior of solutions of \eqref{eq2m}
in the case $-2 < q < 0$ has already been established for more general
uniformly elliptic operators; see Theorem~2 in \cite{felmer2012existence}
and Theorem~4.2 in \cite{mallick2025regularity}.
On the other hand, in the setting of the Laplace equation, this result has been proved in Theorem~3.5 of \cite{dupaigne2007lane}.
The proof of Proposition~\ref{prop4} relies heavily on the construction of appropriate subsolutions and supersolutions.
In this construction, the eigenfunction and its properties are used repeatedly.
For the existence and properties of the eigenfunction, we refer to Theorem~\ref{distance} and \eqref{rem1}.
The function employed in the construction will be a solution of the following second-order ordinary differential equation.
\begin{equation}\label{H}
\left\{
\begin{aligned}
H^{\prime\prime}(t) &= -t^{-\alpha}H^{-\beta}(t),
\quad \text{for all } 0<t\le b<1,\\
H,\, H^{\prime} &> 0
\quad \text{in } (0,b],\\
H(0) &= 0.
\end{aligned}
\right.
\end{equation}

Observe that:
\begin{enumerate}
\item
The function $H$ is concave. Consequently,
\begin{equation}\label{con}
H(t) > t H^{\prime}(t), \qquad \text{for all } 0<t\le b.
\end{equation}

\item
Note that the first equation in \eqref{H} implies
\[
- H^{\prime\prime}(t) \le c^{-\beta} t^{-(\alpha+\beta)},
\qquad \text{for all } 0<t\le b.
\]
Thus, if $\alpha+\beta<1$, it is easy to see that $0 \le H(+0) < \infty$. Consequently, there exist positive constants $c_{1}$ and $c_{2}$ such that
\begin{equation}\label{Lip}
c_{1} t \le H(t) \le c_{2} t,
\qquad \text{for all } 0<t\le b.
\end{equation}

\item
Assume now that $\beta = 1-\alpha$. Then the concavity of $H$ together with the first equation in \eqref{H} yields
\[
- H^{\prime\prime}(t)
< t^{-1} [H^{\prime}(t)]^{\alpha-1},
\qquad \text{for all } 0<t\le b,
\]
that is,
\[
- H^{\prime\prime}(t)\,[H^{\prime}(t)]^{1-\alpha}
< t^{-1},
\qquad \text{for all } 0<t\le b.
\]

Integrating the above inequality over the interval $[s,b]$, we obtain
\[
(H^{\prime})^{2-\alpha}(s) - (H^{\prime})^{2-\alpha}(b)
\le (2-\alpha)(\ln b - \ln s),
\qquad \text{for all } 0<s\le b.
\]

Now choose $\delta_{1}\in(0,b)$ such that
\[
(H^{\prime})^{2-\alpha}(b) + (2-\alpha)\ln b
\le \tfrac{1}{3}(2-\alpha)\bigl[-\ln s\bigr],
\qquad \text{for all } 0<s\le \delta_{1}.
\]
Then we deduce that
\[
H^{\prime}(s)
\le c \bigl[-\ln s\bigr]^{\frac{1}{2-\alpha}},
\qquad \text{for all } s\in(0,\delta_{1}].
\]

Fix $t\in(0,\delta_{1}]$ and integrate the above inequality over $s\in[\varepsilon,t]$ to obtain
\begin{equation}\label{m84}
H(t)-H(\varepsilon)
\le c_{1} t[-\ln t]^{\frac{1}{2-\alpha}}
+ \frac{c_{1}}{2-\alpha}
\int_{\varepsilon}^{t}[-\ln s]^{\frac{\alpha-1}{2-\alpha}}\,ds.
\end{equation}

Observe that
\[
\int_{0}^{t}[-\ln s]^{\frac{\alpha-1}{2-\alpha}}\,ds < +\infty,
\]
and, as a consequence of L'H\^{o}pital's rule,
\[
\lim_{t\to0}
\frac{\displaystyle \int_{0}^{t}[-\ln s]^{\frac{1-\alpha}{2-\alpha}}\,ds}
{t[-\ln t]^{\frac{1}{2-\alpha}}}
=0.
\]
Therefore, letting $\varepsilon\to0$ in \eqref{m84}, we conclude that there exist constants $c_{2}>0$ and $\delta_{2}\in(0,\delta_{1})$ such that
\begin{equation}
H(t)\le c_{2} t[-\ln t]^{\frac{1}{2-\alpha}},
\qquad \text{for all } 0<t\le \delta_{2}.
\end{equation}

Using the above inequality in the first equation of \eqref{H}, we obtain
\[
- H^{\prime\prime}(t)
\ge c_{2}^{\alpha-1} t [-\ln t]^{\frac{\alpha-1}{2-\alpha}},
\qquad \text{for all } 0<t\le \delta_{2}.
\]

Integrating this inequality twice, we deduce that there exist positive constants $c_{3}>0$ and
$\delta_{3}\in(0,\delta_{2})$ such that
\begin{equation}\label{hold}
H(t) \ge c_{3} t [-\ln t]^{\frac{1}{2-\alpha}},
\qquad \text{for all } 0<t\le \delta_{3}.
\end{equation}

\item
Multiplying the first equation in \eqref{H} by $H^{\prime}(t)$ and integrating over $[t,b]$, we obtain:
\begin{equation*}
\begin{aligned}
(H^{\prime})^2(t) - (H^\prime)^2(b)
&= 2 \int_t^b s^{-\alpha} H^{-\beta}(s) H^\prime(s) \, ds \\
&\le 2 t^{-\alpha} \int_{H(t)}^{H(b)} \tau^{-\beta} \, d\tau \\
&\le 2 t^{-\alpha} H^{-\beta}(t) [H(b) - H(t)] \\
&\le 2 t^{-\alpha} H^{-\beta}(t) H(b),
\end{aligned}
\end{equation*}

and hence there exists a constant $C_1>0$ such that
\begin{equation}\label{hp}
H^\prime(t) \le C_1 t^{-\alpha} H^{-\beta}(t),
\qquad \text{for all } 0<t\le b.
\end{equation}
\end{enumerate}
\begin{proposition}\label{prop4}
Let $p \ge 0$ and $0 < q < 2$. Then there exist constants $c_1, c_2 > 0$ and $A > \mathrm{diam}(\Omega)$ such that any solution $u$ of
\begin{equation}\label{eq2}
\left\{
\begin{aligned}
F(D^2 u, Du, u, x) &= \delta(x)^{-q} u^{-p}, && \text{in } \Omega,\\
u &= 0, && \text{on } \partial \Omega,
\end{aligned}
\right.
\end{equation}
satisfies the following boundary estimates:

\begin{enumerate}
    \item[(i)] If $p+q<1$, 
    \[
    c_1 \delta(x) \le u(x) \le c_2 \delta(x), \quad \text{for all } x \in \Omega.
    \]

    \item[(ii)] If $p+q=1$, 
    \[
    c_1 \, \delta(x) \, \log^{\frac{1}{1+p}}\Bigl(\frac{A}{\delta(x)}\Bigr)
    \le u(x) 
    \le c_2 \, \delta(x) \, \log^{\frac{1}{1+p}}\Bigl(\frac{A}{\delta(x)}\Bigr),
    \quad \text{for all } x \in \Omega.
    \]

    \item[(iii)] If $p+q>1$, 
    \[
    c_1 \, \delta(x)^{\frac{2-q}{1+p}} 
    \le u(x) 
    \le c_2 \, \delta(x)^{\frac{2-q}{1+p}}, 
    \quad \text{for all } x \in \Omega.
    \]
\end{enumerate}
\end{proposition}
\begin{proof}
 To prove the estimate, we construct a sub- and supersolution that behaves like a distance function.
Consider the following. If we can show that there exist constants \(m\) and \(M\) such that 
\[
v := m H(c \varphi_1^+) \quad \text{and} \quad w := M H(c \varphi_1^+)
\] 
are a subsolution and a supersolution of \eqref{eq2}, respectively, then the argument follows exactly as in Theorem 3.5 of \cite{dupaigne2007lane}.
  
Let 
\[
v := m H(c \varphi_1^+),
\] 
where $H$ is the solution of \eqref{H} with $\alpha = q$ and $\beta = p$. Then we have
\begin{align*}
D v &= m c H^{\prime}(c \varphi_1^+) D \varphi_1^+, \\
D^2 v &= m c \left[ c H^{\prime\prime}(c \varphi_1^+) D \varphi_1^+ \otimes D \varphi_1^+ + H^{\prime}(c \varphi_1^+) D^2 \varphi_1^+ \right].
\end{align*}

Since $H$, $H^{\prime}$, and $\varphi_1^+$ are nonnegative, we have
\[
[c \varphi_1^+ H^{\prime}(c \varphi_1^+) - H(c \varphi_1^+)]^{-} \le H(c \varphi_1^+).
\]

By using $(H4)$, it follows that
\begin{equation}\label{line3}
\left\{
\begin{aligned}
&F(mc\{H^{\prime}(c\varphi_1^+)D^2\varphi_1^+\},mcH^{\prime}(c\varphi_1^+)D\varphi_1^+,mH(c\varphi_1^+),x)\\
&\leq F(mc\{H^{\prime}(c\varphi_1^+)D^2\varphi_1^+\},mcH^{\prime}(c\varphi_1^+)D(\varphi_1^+),mc\varphi_1^+H^{\prime}(c\varphi_1^+),x)+\gamma m H(c\phi^{+}_{1}).
\end{aligned}
\right.
\end{equation}

Next, we compute
\begin{equation*}
\left\{    
\begin{aligned}
F(D^2 v, Dv, v, x) 
&= F\Big(mc \big[c H^{\prime\prime}(c \varphi_1^+) D \varphi_1^+ \otimes D \varphi_1^+ + H^{\prime}(c \varphi_1^+) D^2 \varphi_1^+\big], mc H^{\prime}(c \varphi_1^+) D \varphi_1^+, m H(c \varphi_1^+), x \Big) \\
&\le F\Big(mc H^{\prime}(c \varphi_1^+) D^2 \varphi_1^+, mc H^{\prime}(c \varphi_1^+) D \varphi_1^+, m H(c \varphi_1^+), x \Big) \\
&\quad + \mathcal{P}_{\lambda,\Lambda}^{+} \Big( mc \, c H^{\prime\prime}(c \varphi_1^+) D \varphi_1^+ \otimes D \varphi_1^+ \Big) \\
&\le F\Big(mc H^{\prime}(c \varphi_1^+) D^2 \varphi_1^+, mc H^{\prime}(c \varphi_1^+) D \varphi_1^+, mc \varphi_1^+ H^{\prime}(c \varphi_1^+), x \Big) \\
&\quad + \gamma m H(c \varphi_1^+) + \mathcal{P}_{\lambda,\Lambda}^{+} \Big( mc^2 H^{\prime\prime}(c \varphi_1^+) D \varphi_1^+ \otimes D \varphi_1^+ \Big) \\
&= mc H^{\prime}(c \varphi_1^+) F(D^2 \varphi_1^+, D \varphi_1^+, \varphi_1^+, x) + \gamma m H(c \varphi_1^+) \\
&\quad + mc^2 H^{\prime\prime}(c \varphi_1^+) \mathcal{P}_{\lambda,\Lambda}^{+} \big( D \varphi_1^+ \otimes D \varphi_1^+ \big) \\
&= mc H^{\prime}(c \varphi_1^+) \mu_1^+ \varphi_1^+ + \gamma m H(c \varphi_1^+) - mc^2 H^{\prime\prime}(c \varphi_1^+) \lambda |D \varphi_1^+|^2 \\
&\le C_1 mc \mu_1^+ \varphi_1^+ (c \varphi_1^+)^{-q} H^{-p}(c \varphi_1^+) + \gamma m H(c \varphi_1^+) \\
&\quad + mc^2 \lambda |D \varphi_1^+|^2 (c \varphi_1^+)^{-q} H^{-p}(c \varphi_1^+) \\
&= m \big[ C_1 c \mu_1^+ \varphi_1^+ + \gamma m (c \varphi_1^+)^q H(c \varphi_1^+)^{1+p} + c^2 \lambda |D \varphi_1^+|^2 \big] (c \varphi_1^+)^{-q} H^{-p}(c \varphi_1^+) \\
&\le m \big[ k_1 \varphi_1^+ + k_2 |D \varphi_1^+|^2 + k_3 \big] \delta(x)^{-q} H^{-p}(c \varphi_1^+),
\end{aligned}
\right.
\end{equation*}
Here, we define 
\[
k_{3} = \|(c \varphi_1^+)^q H(c \varphi_1^+)^{1+p}\|_{L^\infty(\Omega)},
\] 
and we have used \eqref{line3} in the second inequality, \eqref{hp} in the third-to-last line, and \eqref{ineigen} in the last line.  

Thus, \(v\) is a subsolution, and by the comparison principle, we obtain
\[
u \ge v = m H(c \varphi_1^+) \ge \tilde{c} \varphi_1^+ \ge c_1 \delta(x),
\]  
where we used \eqref{ineigen} and \eqref{Lip}. This establishes half of the inequality.  

In order to prove the rest of the inequality, we consider 
\[
w := M H(c \varphi_1^+),
\] 
where \(H\) again satisfies \eqref{H}. We will also use the inequality
\[
H(c \varphi_1^+) \ge H'(c \varphi_1^+) \, c \varphi_1^+,
\] 
which is a consequence of the concavity of \(H\).

We begin with the following observation.\\

Next, Let $w=MH(c\varphi_1^+).$ Then
\begin{equation*}
\left\{
\begin{aligned}
F(D^{2}w,Dw,w,x) 
&= F\Big(Mc\{c H^{\prime\prime}(c\varphi_1^+) D(\varphi_1^+) \otimes D(\varphi_1^+)\} 
+ Mc\{H^{\prime}(c\varphi_1^+) D^2\varphi_1^+\}, \\
&\quad Mc H^{\prime}(c\varphi_1^+) D(\varphi_1^+), MH(c\varphi_1^+), x\Big) \\
&\geq F\Big(Mc\{H^{\prime}(c\varphi_1^+) D^2\varphi_1^+\}, 
Mc H^{\prime}(c\varphi_1^+) D(\varphi_1^+), MH(c\varphi_1^+), x\Big) \\
&\quad + \mathcal{P}^{-}_{\lambda,\Lambda}\Big(Mc\{c H^{\prime\prime}(c\varphi_1^+) D(\varphi_1^+) \otimes D(\varphi_1^+)\}\Big) \\
&\geq F\Big(Mc\{H^{\prime}(c\varphi_1^+) D^2\varphi_1^+\}, 
Mc H^{\prime}(c\varphi_1^+) D(\varphi_1^+), Mc H^{\prime}(c\varphi_1^+) \varphi_1^+, x\Big) \\
&\quad + \mathcal{P}^{-}_{\lambda,\Lambda}\Big(Mc\{c H^{\prime\prime}(c\varphi_1^+) D(\varphi_1^+) \otimes D(\varphi_1^+)\}\Big) \\
&= Mc H^{\prime}(c\varphi_1^+) F(D^2\varphi_1^+, D\varphi_1^+, \varphi_1^+, x) 
+ Mc^{2} \Lambda |D\varphi_1^+|^2 (-H^{\prime\prime}(c\varphi_1^+)) \\
&= Mc H^{\prime}(c\varphi_1^+) \mu^{+}_{1} \varphi_1^+ 
+ Mc^{2} \Lambda |D\varphi_1^+|^2 (c\varphi_1^+)^{-\alpha} H^{-\beta}(c\varphi_1^+) \\
&\geq Mc^{2} \Lambda |D\varphi_1^+|^2 (c\varphi_1^+)^{-\alpha} H^{-\beta}(c\varphi_1^+) \\
&\geq M k \, d(x)^{-\alpha} H^{-\beta}(c\varphi_1^+),
\end{aligned}
\right.
\end{equation*}
Choose $M>0$ large enough such that $Mk>1$ in $\Omega.$ This implies that $w=MH(c\varphi_1^+)$ is supersolution of \eqref{eq2}. Therefore, by comparison principle we have the second half of the inequality.\\
\end{proof}
\begin{remark}\label{brem}
Consider the following equation:
\begin{equation}\label{req2}
\left\{
\begin{aligned}
F(D^2u,Du,u,x) &= f(x) u^{-p} && \text{in } \Omega,\\
u &= 0 && \text{on } \partial\Omega,
\end{aligned}
\right.
\end{equation}
where $c_1 \delta(x)^{-q} \leq f(x) \leq c_2 \delta(x)^{-q}$ for some positive constants $c_1$ and $c_2$.  
Since the proof of Proposition \ref{prop4} relies on the comparison principle and the construction of sub- and supersolutions, 
the conclusions $(i)$–$(iii)$ of Proposition \ref{prop4} also hold for any solution of \eqref{req2}, depending on the relation between the exponents $p$ and $q$.
\end{remark}
\begin{theorem}\label{th9}
Let $0<a<1,$ $A>diam(\Omega),~p\geq0$ and $q>0$ be such that $p+q=1.$ Then the problem 
\begin{equation}\label{exeq3}
\left\{
\begin{aligned}
F(D^2u(x),Du(x),u(x),x) &= \delta(x)^{-q}log^{-a}\left(\frac{A}{\delta(x)}\right)u^{-p}~~\text{in}~~\Omega,\\
u &= 0~~\text{on}~~\partial\Omega,
\end{aligned}
\right.
\end{equation}
has unique solution $u$ which satisfies 
$$C_1\delta(x)log^{\frac{1-a}{1+p}}\left(\frac{A}{\delta(x)}\right)\leq u(x)\leq C_2\delta(x)log^{\frac{1-a}{1+p}}\left(\frac{A}{\delta(x)}\right)$$
for some $C_1~,C_2>0.$
\end{theorem}
\begin{proof}
   Let $w(x)=\varphi_1^{+}(x)\log^{b}\left(\frac{A}{\varphi_1^{+}(x)}\right),$ 
where $\varphi_1^{+}$ is the eigenfunction from Theorem \ref{distance} 
and $b=\frac{1-a}{1+p}\in(0,1)$. 
Now let us calculate 
\begin{equation*}
\left\{
\begin{aligned}
Dw &= \Big[ - b \, \log^{b-1}\!\left(\frac{A}{\varphi_1^{+}(x)}\right) 
       + \log^{b}\!\left(\frac{A}{\varphi_1^{+}(x)}\right) \Big] 
       D\varphi_1^{+}(x),\\[2mm]
D^2 w &= \Big[ - b \, \log^{b-1}\!\left(\frac{A}{\varphi_1^{+}(x)}\right) 
              + \log^{b}\!\left(\frac{A}{\varphi_1^{+}(x)}\right) \Big] D^2 \varphi_1^{+}(x) \\ 
      &\quad + \Big[ \frac{- b(1-b)}{\varphi_1^{+}} \log^{b-2}\!\left(\frac{A}{\varphi_1^{+}(x)}\right) 
                      - \frac{b}{\varphi_1^{+}} \log^{b-1}\!\left(\frac{A}{\varphi_1^{+}(x)}\right) \Big] 
        D\varphi_1^{+} \otimes D\varphi_1^{+} \\[1mm]
&= \alpha \, D^2 \varphi_1^{+} - \beta \, D\varphi_1^{+} \otimes D\varphi_1^{+},
\end{aligned}
\right.
\end{equation*}
where
\begin{equation*}
\left\{
\begin{aligned}
\alpha &= \Big[ - b \, \log^{b-1}\!\Big(\frac{A}{\varphi_1^{+}(x)}\Big) + \log^{b}\!\Big(\frac{A}{\varphi_1^{+}(x)}\Big) \Big] \\
       &= \log^{b-1}\!\Big(\frac{A}{\varphi_1^{+}(x)}\Big) \Big[ -b + \log\Big(\frac{A}{\varphi_1^{+}(x)}\Big) \Big],\\[1mm]
\beta  &= - \Big[ \frac{-b(1-b)}{\varphi_1^{+}} \log^{b-2}\!\Big(\frac{A}{\varphi_1^{+}}\Big) 
                  - \frac{b}{\varphi_1^{+}} \log^{b-1}\!\Big(\frac{A}{\varphi_1^{+}}\Big) \Big] \\
       &= \frac{b(1-b)}{\varphi_1^{+}} \log^{b-2}\!\Big(\frac{A}{\varphi_1^{+}}\Big) 
          + \frac{b}{\varphi_1^{+}} \log^{b-1}\!\Big(\frac{A}{\varphi_1^{+}}\Big).
\end{aligned}
\right.
\end{equation*}
We can choose $A$ sufficiently large such that $\log\left(\frac{A}{\varphi_1^{+}}\right)>1.$ With this assumption, we find that
\[\alpha>0~~\text{and}~~\beta>0.\]
Now, we compute 
\begin{equation*}
\left\{
\begin{aligned}
F(D^2 w(x), D w(x), w(x), x) 
&= F\Big(\alpha D^2 \varphi_1^{+} - \beta D \varphi_1^{+} \otimes D \varphi_1^{+},\, \alpha D \varphi_1^{+},\, \varphi_1^{+}(x) \log^b \Big(\frac{A}{\varphi_1^{+}}\Big),\, x \Big) \\
&\geq F\Big(\alpha D^2 \varphi_1^{+},\, \alpha D \varphi_1^{+},\, \varphi_1^{+}(x) \log^b \Big(\frac{A}{\varphi_1^{+}}\Big),\, x \Big) 
  + \mathcal{P}_{\lambda, \Lambda}^{-}\Big(-\beta D \varphi_1^{+} \otimes D \varphi_1^{+}\Big) \\
&\geq F\Big(\alpha D^2 \varphi_1^{+},\, \alpha D \varphi_1^{+},\, \alpha \varphi_1^{+} + b \varphi_1^{+}(x) \log^{b-1} \Big(\frac{A}{\varphi_1^{+}}\Big),\, x \Big)
  + \beta \lambda |\nabla \varphi_1^{+}|^2 \\
&\geq F\Big(\alpha D^2 \varphi_1^{+},\, \alpha D \varphi_1^{+},\, \alpha \varphi_1^{+},\, x \Big) + \beta \lambda |\nabla \varphi_1^{+}|^2
  ~~~(\text{Using monotonicity in the $u$-variable}) \\
&= \alpha \mu_1^{+} \varphi_1^{+} + \beta \lambda |\nabla \varphi_1^{+}|^2.
\end{aligned}
\right.
\end{equation*}
\text{Substituting the expressions for } $\alpha \text{ and } \beta$, \text{ we obtain:}
\begin{equation*}
\left\{
\begin{aligned}
F(D^2 w(x), D w(x), w(x), x) 
&\geq 
\log^{b-1} \Big(\frac{A}{\varphi_1^{+}}\Big) \Big[ -b + \log \Big(\frac{A}{\varphi_1^{+}}\Big) \Big] \mu_1^{+} \varphi_1^{+} \\
&\quad + \Big[ \frac{b(1-b)}{\varphi_1^{+}} \log^{b-2} \Big(\frac{A}{\varphi_1^{+}}\Big) + \frac{b}{\varphi_1^{+}} \log^{b-1} \Big(\frac{A}{\varphi_1^{+}}\Big) \Big] \lambda |\nabla \varphi_1^{+}|^2 \\
&= (\varphi_1^{+})^{-1} \log^{b-1} \Big(\frac{A}{\varphi_1^{+}}\Big) 
   \Big[ \Big(-b + \log \Big(\frac{A}{\varphi_1^{+}(x)}\Big)\Big) \mu_1^{+} (\varphi_1^{+})^2 \\
&\quad + b(1-b) \log^{-1} \Big(\frac{A}{\varphi_1^{+}}\Big) \lambda |\nabla \varphi_1^{+}|^2 \Big] \\
&\geq C_1 (\varphi_1^{+})^{-1} \log^{b-1} \Big(\frac{A}{\varphi_1^{+}}\Big),
\end{aligned}
\right.
\end{equation*}
and 
\begin{equation*}
\left\{
\begin{aligned}
F(D^2 w(x), D w(x), w(x), x) 
&= F\Big(\alpha D^{2} \varphi_1^{+} - \beta D \varphi_1^{+} \otimes D \varphi_1^{+}, \alpha D \varphi_1^{+}, \varphi_1^{+}(x) \log^b \Big(\frac{A}{\varphi_1^{+}}\Big), x \Big) \\
&\leq F\Big(\alpha D^{2} \varphi_1^{+}, \alpha D \varphi_1^{+}, \varphi_1^{+}(x) \log^b \Big(\frac{A}{\varphi_1^{+}}\Big), x \Big) + \mathcal{P}_{\lambda, \Lambda}^{+}\big(-\beta D \varphi_1^{+} \otimes D \varphi_1^{+}\big) \\
&\leq F\Big(\alpha D^{2} \varphi_1^{+}, \alpha D \varphi_1^{+}, \alpha \varphi_1^{+} + b \varphi_1^{+} \log^{b-1} \Big(\frac{A}{\varphi_1^{+}}\Big), x \Big) + \beta \Lambda |\nabla \varphi_1^{+}|^2 \\
&\leq F\Big(\alpha D^{2} \varphi_1^{+}, \alpha D \varphi_1^{+}, \alpha \varphi_1^{+}, x \Big) + \mathcal{P}_{\lambda, \Lambda}^{+}(0) + \Gamma |0| \\
&\quad + \gamma \Big(-b \log^{b-1} \Big(\frac{A}{\varphi_1^{+}}\Big) \varphi_1^{+} \Big)^{-} + \beta \Lambda |\nabla \varphi_1^{+}|^2 \\
&= \alpha \mu_1^{+} \varphi_1^{+} + \gamma b \log^{b-1} \Big(\frac{A}{\varphi_1^{+}}\Big) \varphi_1^{+} + \beta \Lambda |\nabla \varphi_1^{+}|^2.
\end{aligned}
\right.
\end{equation*}
\text{Substituting } $\alpha \text{ and } \beta$ \text{ into the above expression, we get:}
\begin{equation*}
\left\{
\begin{aligned}
F(D^2 w(x), D w(x), w(x), x) 
&\leq \log^{b-1}\Big(\frac{A}{\varphi_1^{+}}\Big) 
\Big[-b + \log\Big(\frac{A}{\varphi_1^{+}}\Big)\Big] \mu_1^{+} \varphi_1^{+} \\
&\quad + \gamma b \, \log^{b-1}\Big(\frac{A}{\varphi_1^{+}}\Big) \varphi_1^{+} \\
&\quad + \Big[ \frac{b(1-b)}{\varphi_1^{+}} \log^{b-2}\Big(\frac{A}{\varphi_1^{+}}\Big) 
+ \frac{b}{\varphi_1^{+}} \log^{b-1}\Big(\frac{A}{\varphi_1^{+}}\Big) \Big] |\nabla \varphi_1^{+}|^2 \Lambda \\
&= (\varphi_1^{+})^{-1} \log^{b-1}\Big(\frac{A}{\varphi_1^{+}}\Big)
\Big[ \big\{-b + \log\Big(\frac{A}{\varphi_1^{+}}\Big) \big\} \mu_1^{+} (\varphi_1^{+})^2 \\
&\quad + \gamma b (\varphi_1^{+})^2 + b(1-b) \log^{-1} \Big(\frac{A}{\varphi_1^{+}}\Big) |\nabla \varphi_1^{+}|^2 \Lambda \Big] \\
&\leq C_2 (\varphi_1^{+})^{-1} \log^{b-1}\Big(\frac{A}{\varphi_1^{+}}\Big),
\end{aligned}
\right.
\end{equation*}
Thus, there exist constants $C_1, C_2 > 0$ such that
\begin{equation*}
C_1 (\varphi_1^{+})^{-1} \log^{b-1}\Big(\frac{A}{\varphi_1^{+}}\Big)
\leq F(D^2 w(x), D w(x), w(x), x)
\leq C_2 (\varphi_1^{+})^{-1} \log^{b-1}\Big(\frac{A}{\varphi_1^{+}}\Big),
\end{equation*}
and equivalently,
\begin{equation*}
C_1 (\varphi_1^{+})^{-q} \log^{-a}\Big(\frac{A}{\varphi_1^{+}}\Big) w^{-p} 
\leq F(D^2 w(x), D w(x), w(x), x)
\leq C_2 (\varphi_1^{+})^{-q} \log^{-a}\Big(\frac{A}{\varphi_1^{+}}\Big) w^{-p}.
\end{equation*}
Thus, $u_1 = m w$ and $u_2 = M w$ are subsolution and supersolution of \eqref{exeq3} for suitable constants $0 < m < 1 < M$. Hence, Problem \eqref{exeq3} admits a solution $u$ satisfying
\[
m \, \varphi_1^{+} \, \log^{\frac{1-a}{1+p}}\!\Big(\frac{A}{\varphi_1^{+}}\Big)
\leq u \leq 
M \, \varphi_1^{+} \, \log^{\frac{1-a}{1+p}}\!\Big(\frac{A}{\varphi_1^{+}}\Big)
\quad \text{in } \Omega.
\]
\textbf{Uniqueness:} 
Observe that for each fixed $x \in \Omega$, the function
\[
t \in (0, \infty) \mapsto \delta(x)^{-q} \log^{-a} \Big( \frac{A}{\delta(x)} \Big) t^{-p}
\]
is decreasing in $t$. Hence, the operator considered here is proper, and uniqueness follows.  

More precisely, let $u$ and $v$ be two distinct solutions of \eqref{exeq3}. Without loss of generality, assume that 
\[
\omega := \{ x \in \Omega \mid v(x) < u(x) \}
\]
is a non-empty open set and satisfies $\omega \subset\subset \Omega$.
Thus, in $\omega$ we have
\begin{equation*}
\left\{
\begin{aligned}
 F(D^2 u(x), Du(x), u(x), x) 
&= \delta(x)^{-q} \log^{-a} \left( \frac{A}{\delta(x)} \right) u^{-p} \\
&\leq \delta(x)^{-q} \log^{-a} \left( \frac{A}{\delta(x)} \right) v^{-p}\\
&= F(D^2 v(x), Dv(x), v(x), x).
\end{aligned}
\right.
\end{equation*}
and $v=u$ on $\partial\omega.$ Thus $u\leq v$ in $\omega,$ which contradicts the definition of $\omega.$ Thus $u\equiv v.$
\end{proof}
\begin{corollary}\label{coro2}
Let $C>0$ and let $a, A, p, q$ be as in Theorem \ref{th9}. Then there exists a constant $\underline{c}>0$ such that any solution $u$ of 
\begin{equation*}
\left\{
\begin{aligned}
F(D^2 u, Du, u, x) &\geq C \, \delta(x)^{-q} \log^{-a}\left(\frac{A}{\delta(x)}\right) u^{-p} && \text{in } \Omega,\\
u &> 0 && \text{in } \Omega,\\
u &= 0 && \text{on } \partial\Omega,
\end{aligned}
\right.
\end{equation*}
satisfies
\[
u(x) \geq \underline{c} \, \delta(x) \, \log^{\frac{1-a}{1+p}}\left(\frac{A}{\delta(x)}\right) \quad \text{in } \Omega.
\]
\end{corollary}
\begin{proposition}\label{prop11}
Let $A > 6 \, \mathrm{diam}(\Omega)$ and $C > 0$. Then there exists a constant $k > 0$ such that any solution 
$u $ of 
\begin{equation}\label{exeq4}
\left\{
\begin{aligned}
F(D^2 u, Du, u, x) &\geq C \, \delta(x)^{-1} \log^{-1}\left(\frac{A}{\delta(x)}\right) && \text{in } \Omega,\\
u &> 0 && \text{in } \Omega,\\
u &= 0 && \text{on } \partial \Omega,
\end{aligned}
\right.
\end{equation}
satisfies
\[
u(x) \geq k \, \delta(x) \, \log \left[ \log \left(\frac{A}{\delta(x)}\right) \right] \quad \text{in } \Omega.
\]
\end{proposition}
\begin{proof}
Let 
\[
w(x) = \varphi_1^{+}(x) \, \log\Bigg[\log\Big(\frac{A}{\varphi_1^{+}(x)}\Big)\Bigg],
\] 
where $\varphi_1^{+}$ is the positive eigenfunction from Theorem \ref{distance}.  
Then, proceeding as in Proposition \ref{th9}, we compute
\begin{equation*}
\left\{
\begin{aligned}
Dw &= - \Big[\log\Big(\frac{A}{\varphi_1^{+}}\Big)\Big]^{-1} D\varphi_1^{+},\\[2mm]
D^2 w &= \Bigg[ \log\Big(\log\Big(\frac{A}{\varphi_1^{+}}\Big)\Big) - \log^{-1}\Big(\frac{A}{\varphi_1^{+}}\Big) \Bigg] D^2 \varphi_1^{+} \\
&\quad + \Bigg[ - \frac{1}{\varphi_1^{+}} \log^{-1}\Big(\frac{A}{\varphi_1^{+}}\Big) - \frac{1}{\varphi_1^{+}} \log^{-2}\Big(\frac{A}{\varphi_1^{+}}\Big) \Bigg] D\varphi_1^{+} \otimes D\varphi_1^{+} \\[1mm]
&= \alpha D^2 \varphi_1^{+} - \beta D\varphi_1^{+} \otimes D\varphi_1^{+},
\end{aligned}
\right.
\end{equation*}
where
\begin{equation*}
\left\{
\begin{aligned}
\alpha &= \log\Big(\log\big(\frac{A}{\varphi_1^{+}}\big)\Big) - \log^{-1}\Big(\frac{A}{\varphi_1^{+}}\Big),\\[1mm]
\beta &= - \Bigg[ - \frac{1}{\varphi_1^{+}} \log^{-1}\Big(\frac{A}{\varphi_1^{+}}\Big) - \frac{1}{\varphi_1^{+}} \log^{-2}\Big(\frac{A}{\varphi_1^{+}}\Big) \Bigg].
\end{aligned}
\right.
\end{equation*}
With a sufficiently large choice of $A$, it is easy to see that
\[
\alpha > 0 \quad \text{and} \quad \beta > 0.
\]
\begin{equation*}
\begin{aligned}
F(D^2w, Dw, w, x) &= F\Big(\alpha D^2 \varphi_1^+ - \beta D \varphi_1^+ \otimes D \varphi_1^+, 
                        -\log^{-1}\Big(\frac{A}{\varphi_1^+}\Big) D \varphi_1^+, 
                        \varphi_1^+ \log\Big[\log\Big(\frac{A}{\varphi_1^+}\Big)\Big], x \Big) \\
&\leq F\Big(\alpha D^2 \varphi_1^+, -\log^{-1}\Big(\frac{A}{\varphi_1^+}\Big) D \varphi_1^+, 
          \varphi_1^+ \log\Big[\log\Big(\frac{A}{\varphi_1^+}\Big)\Big], x\Big) 
   + \mathcal{P}_{\lambda,\Lambda}^+\Big(-\beta D\varphi_1^+ \otimes D\varphi_1^+\Big) \\
&= F\Big(\alpha D^2 \varphi_1^+, \alpha D\varphi_1^+ - \log\Big[\log\Big(\frac{A}{\varphi_1^+}\Big)\Big] D\varphi_1^+, 
          \alpha \varphi_1^+ + \log^{-1}\Big(\frac{A}{\varphi_1^+}\Big)\varphi_1^+, x\Big) 
   + \beta \Lambda |\nabla \varphi_1^+|^2 \\
&\leq F\Big(\alpha D^2 \varphi_1^+, \alpha D\varphi_1^+, \alpha \varphi_1^+, x \Big) 
   + \Gamma \Big|\!-\log\big(\log(\frac{A}{\varphi_1^+})\big)\Big| 
   + \gamma \Big(\log^{-1}\Big(\frac{A}{\varphi_1^+}\Big) \varphi_1^+\Big)^{-} 
   + \beta \Lambda |\nabla \varphi_1^+|^2 \\
&= \alpha F(D^2 \varphi_1^+, D\varphi_1^+, \varphi_1^+, x) 
   + \Gamma \log\Big(\log\Big(\frac{A}{\varphi_1^+}\Big)\Big) 
   + \gamma \log^{-1}\Big(\frac{A}{\varphi_1^+}\Big) \varphi_1^+ 
   + \beta \Lambda |\nabla \varphi_1^+|^2 \\
&= \alpha \mu_1^+ \varphi_1^+ 
   + \Gamma \log\Big(\log\Big(\frac{A}{\varphi_1^+}\Big)\Big) 
   + \gamma \log^{-1}\Big(\frac{A}{\varphi_1^+}\Big) \varphi_1^+ 
   + \beta \Lambda |\nabla \varphi_1^+|^2 \\
&= \Big[\log\Big(\log\Big(\frac{A}{\varphi_1^+}\Big)\Big) - \log^{-1}\Big(\frac{A}{\varphi_1^+}\Big)\Big] \mu_1^+ \varphi_1^+ 
   + \Gamma \log\Big(\log\Big(\frac{A}{\varphi_1^+}\Big)\Big) 
   + \gamma \log^{-1}\Big(\frac{A}{\varphi_1^+}\Big) \varphi_1^+ \\
&\quad + \Big[\frac{1}{\varphi_1^+} \log^{-1}\Big(\frac{A}{\varphi_1^+}\Big) 
           + \frac{1}{\varphi_1^+} \log^{-2}\Big(\frac{A}{\varphi_1^+}\Big)\Big] |\nabla \varphi_1^+|^2 \Lambda \\
&\leq c_0 \Big[\varphi_1^+ \log\Big(\frac{A}{\varphi_1^+}\Big)\Big]^{-1},
\end{aligned}
\end{equation*}
where
\[
\begin{aligned}
c_0 = \Bigg\| 
& \Big[
   \varphi_1^{+}
   \log\!\big(\log(\tfrac{A}{\varphi_1^{+}})\big)
   \log\!\big(\tfrac{A}{\varphi_1^{+}(x)}\big)
   - \varphi_1^{+}
  \Big] 
  \mu_1^{+}\varphi_1^{+} \\[6pt]
&\quad +\,
  \Gamma\, 
  \log\!\big(\log(\tfrac{A}{\varphi_1^{+}})\big)
  \log\!\big(\tfrac{A}{\varphi_1^{+}(x)}\big)\,
  \varphi_1^{+} \\[6pt]
&\quad +\,
  \gamma\,(\varphi_1^{+})^2
  + 
  \Big[1+\log^{-1}\!\big(\tfrac{A}{\varphi_1^{+}(x)}\big)\Big]
  |\nabla \varphi_1^{+}|^2\,\Lambda
 \;\Bigg\|_{L^{\infty}(\Omega)} .
\end{aligned}
\]
Moreover, in view of \eqref{ineigen}, we can find $m>0$ such that
\begin{equation}\label{wdeltaineq}
\Big[\varphi_1^+ \log\Big(\frac{A}{\varphi_1^+}\Big)\Big]^{-1} 
\leq m \Big[\delta(x) \log\Big(\frac{A}{\delta(x)}\Big)\Big]^{-1}.
\end{equation}
Thus, 
\[
F(D^2w,Dw,w,x) \leq c_{0} m \Big[\delta(x) \log\Big(\frac{A}{\delta(x)}\Big)\Big]^{-1}.
\]

Set $\underline{k} = C (c_{0} m)^{-1}$, then $\underline{w} = \underline{k} w$ satisfies
\begin{equation*}
\left\{
\begin{aligned}
F(D^2 \underline{w}, D \underline{w}, \underline{w}, x) 
&= F(\underline{k} D^2 w, \underline{k} D w, \underline{k} w, x) \\
&= \underline{k} F(D^2 w, D w, w, x)  \\
&\leq C \Big[\delta(x) \log\Big(\frac{A}{\delta(x)}\Big)\Big]^{-1}.
\end{aligned}
\right.
\end{equation*}
Thus from the comparison principle, we find that 
\begin{equation}
u\geq \underline{w}\geq k \delta(x)\Big[\log\log\left(\frac{A}{\delta(x)}\right)\Big],
\end{equation}
for some $k>0.$ In the last line we have used \eqref{ineigen}.
\end{proof}
\begin{theorem}\label{th4}
Let $p \geq 0$, $A > \operatorname{diam}(\Omega)$, and $a \in \mathbb{R}$. Then the problem
\begin{equation}\label{exeq6}
\left\{
\begin{aligned}
F(D^2u,Du,u,x) &= \delta(x)^{-2}\log^{-a}\!\left(\frac{A}{\delta(x)}\right)u^{-p}
    && \text{in } \Omega,\\
u &> 0 && \text{in } \Omega,\\
u &= 0 && \text{on } \partial\Omega,
\end{aligned}
\right.
\end{equation}
admits a solution if and only if $a>1$.  Moreover, if $a>1$, then problem \eqref{exeq6} has a unique solution $u$, and there exist constants
$c_1, c_2 > 0$ such that
\[
c_1 \log^{\frac{1-a}{1+p}}\!\left(\frac{A}{\delta(x)}\right)
\leq u(x) \leq
c_2 \log^{\frac{1-a}{1+p}}\!\left(\frac{A}{\delta(x)}\right)
\qquad \text{in } \Omega.
\]
\end{theorem}
\begin{proof}
Fix $B>A$ and define a decreasing function $k:(0,B)\to\mathbb{R}$ by
\[
k(t)=t^{-2}\log^{-a}\!\left(\frac{B}{t}\right).
\]
With this notation, equation \eqref{exeq6} can be rewritten as
\[
\left\{
\begin{aligned}
F(D^2u,Du,u,x) &\geq c\,k(\delta(x))\,u^{-p}
    && \text{in } \Omega,\\
u &> 0 && \text{in } \Omega,\\
u &= 0 && \text{on } \partial\Omega.
\end{aligned}
\right.
\]
where $c>0$. By Theorem~\ref{th1}, the condition
\[
\int_{0}^{A} t\,k(t)\,dt < \infty
\]
holds if and only if $a>1$.

Now assume that $a>1$ and define
\[
w(x)=\log^{b}\!\left(\frac{B}{\varphi_1^{+}(x)}\right),
\]
where
\[
b=\frac{1-a}{1+p}\in(-\infty,0).
\]
Then
\begin{equation*}
\left\{
\begin{aligned}
Dw
&=\left[-\frac{b}{\varphi_1^{+}(x)}
\log^{\,b-1}\!\left(\frac{B}{\varphi_1^{+}(x)}\right)\right]
D\varphi_1^{+}(x)\\
&= \alpha\, D\varphi_1^{+}(x),\\[6pt]
D^2w
&=\left[-\frac{b}{\varphi_1^{+}(x)}
\log^{\,b-1}\!\left(\frac{B}{\varphi_1^{+}(x)}\right)\right]
D^2\varphi_1^{+}(x)\\
&\quad+\left[
\frac{b(b-1)}{(\varphi_1^{+}(x))^{2}}
\log^{\,b-2}\!\left(\frac{B}{\varphi_1^{+}(x)}\right)
+\frac{b}{(\varphi_1^{+}(x))^{2}}
\log^{\,b-1}\!\left(\frac{B}{\varphi_1^{+}(x)}\right)
\right]
D\varphi_1^{+}\otimes D\varphi_1^{+}\\
&= \alpha\, D^2\varphi_1^{+}
-\beta\, D\varphi_1^{+}\otimes D\varphi_1^{+}.
\end{aligned}
\right.
\end{equation*}
where
\begin{equation*}
\left\{
\begin{aligned}
\alpha
&=-\frac{b}{\varphi_1^{+}(x)}
\log^{\,b-1}\!\left(\frac{B}{\varphi_1^{+}(x)}\right)
>0,\\[6pt]
\beta
&=-\frac{b}{\big(\varphi_1^{+}(x)\big)^2}
\log^{\,b-1}\!\left(\frac{B}{\varphi_1^{+}(x)}\right)
\left[
\frac{(b-1)+\log\!\left(\frac{B}{\varphi_1^{+}(x)}\right)}
{\log\!\left(\frac{B}{\varphi_1^{+}(x)}\right)}
\right].
\end{aligned}
\right.
\end{equation*}
Now choose $B>0$ sufficiently large such that
\[
\log\!\left(\frac{B}{\varphi_1^{+}(x)}\right)\geq 2(1-b)
\qquad \text{in } \Omega.
\]
Then $\beta>0$. Moreover,
\begin{equation*}
\left\{
\begin{aligned}
F(D^2w,Dw,w,x)
&=F\!\left(
\alpha D^{2}\varphi^{+}_{1}
-\beta D\varphi^{+}_{1}\otimes D\varphi^{+}_{1},
\alpha D\varphi_1^{+},
\log^b\!\left(\frac{B}{\varphi_1^{+}}\right),
x
\right)\\
&\leq
F\!\left(
\alpha D^{2}\varphi^{+}_{1},
\alpha D\varphi_1^{+},
\log^b\!\left(\frac{B}{\varphi_1^{+}}\right),
x
\right)
+\mathcal{P}_{\lambda,\Lambda}^{+}
\!\left(-\beta D\varphi^{+}_{1}\otimes D\varphi^{+}_{1}\right)\\
&=
F\!\left(
\alpha D^{2}\varphi^{+}_{1},
\alpha D\varphi_1^{+},
\alpha\varphi_1^{+},
x
\right)
+\Big[
F\!\left(
\alpha D^{2}\varphi^{+}_{1},
\alpha D\varphi_1^{+},
\log^b\!\left(\frac{B}{\varphi_1^{+}}\right),
x
\right)\\
&\hspace{4em}
- F\!\left(
\alpha D^{2}\varphi^{+}_{1},
\alpha D\varphi_1^{+},
\alpha\varphi_1^{+},
x
\right)
\Big]
+\beta\Lambda|\nabla\varphi_1^{+}|^2\\
&\leq
F\!\left(
\alpha D^{2}\varphi^{+}_{1},
\alpha D\varphi_1^{+},
\alpha\varphi_1^{+},
x
\right)\\
&\quad
+\gamma\Big(
\alpha\varphi_1^{+}
-\log^b\!\left(\frac{B}{\varphi_1^{+}}\right)
\Big)^{-}
+\beta\Lambda|\nabla\varphi_1^{+}|^2\\
&=
\alpha F\!\left(
D^{2}\varphi^{+}_{1},
D\varphi_1^{+},
\varphi_1^{+},
x
\right)
+\gamma\Big(
\alpha\varphi_1^{+}
-\log^b\!\left(\frac{B}{\varphi_1^{+}}\right)
\Big)^{-}
+\beta\Lambda|\nabla\varphi_1^{+}|^2.
\end{aligned}
\right.
\end{equation*}
Since
\[
\alpha\varphi_1^{+}-\log^b\!\left( \frac{B}{\varphi_1^{+}}\right)
=-b\,\log^{\,b-1}\!\left(\frac{B}{\varphi_1^+}\right)-w
\geq -w,
\]
it follows that
\[
\left(\alpha\varphi_1^{+}-\log^b\!\left( \frac{B}{\varphi_1^{+}}\right)\right)^{-}
\leq w.
\]

Therefore,
\begin{equation*}
\left\{
\begin{aligned}
F(D^2w,Dw,w,x)
&\leq \alpha\mu_1^{+}\varphi_1^{+}
+\gamma w
+\beta\Lambda|\nabla\varphi_1^{+}|^2\\
&= \alpha\mu_1^{+}\varphi_1^{+}
+\gamma\log^b\!\left(\frac{B}{\varphi_1^{+}}\right)
+\beta\Lambda|\nabla\varphi_1^{+}|^2.
\end{aligned}
\right.
\end{equation*}
Substituting the expressions of $\alpha$ and $\beta$ into the previous inequality, we obtain
\begin{equation*}
\left\{
\begin{aligned}
F(D^2w,Dw,w,x)
&\leq 
\left[-\frac{b}{\varphi_1^{+}(x)}\log^{b-1}\!\left(\frac{B}{\varphi_1^{+}(x)}\right)\right]
\mu_1^{+}\varphi_1^{+}(x)\\
&\quad +\gamma \log^{b}\!\left(\frac{B}{\varphi_1^{+}(x)}\right)\\
&\quad +\left[-\frac{b}{(\varphi_1^{+}(x))^{2}}
\log^{b-1}\!\left(\frac{B}{\varphi_1^{+}(x)}\right)
\frac{(b-1)+\log\!\left(\frac{B}{\varphi_1^{+}(x)}\right)}
{\log\!\left(\frac{B}{\varphi_1^{+}(x)}\right)}\right]
\Lambda|\nabla\varphi_1^{+}(x)|^{2}.
\end{aligned}
\right.
\end{equation*}
Factoring out $(\varphi_1^{+})^{-2}
\log^{\,b-1}\!\left(\frac{B}{\varphi_1^{+}}\right),$  we obtain
 \begin{equation*}
F(D^2w,Dw,w,x)
\leq
C_2\,(\varphi_1^{+})^{-2}
\log^{\,b-1}\!\left(\frac{B}{\varphi_1^{+}}\right)
\qquad \text{in } \Omega,
\end{equation*}
where
\[
C_2
=
\Big\|
-b\mu_1^{+}(\varphi_1^{+})^{2}
+\gamma (\varphi_1^{+})^{2}\log\!\left(\frac{B}{\varphi_1^{+}}\right)
-b\,\frac{(b-1)+\log\!\left(\frac{B}{\varphi_1^{+}}\right)}
{\log\!\left(\frac{B}{\varphi_1^{+}}\right)}
|\nabla\varphi_1^{+}|^{2}
\Big\|_{L^{\infty}(\Omega)}.
\]
\begin{equation*}
\left\{
\begin{aligned}
F(D^2w,Dw,w,x)
&=F\!\left(
\alpha D^{2}\varphi^{+}_{1}-\beta D\varphi^{+}_{1}\otimes D\varphi^{+}_{1},
\alpha D\varphi_1^{+},
\log^{b}\!\left(\frac{B}{\varphi_1^{+}}\right),
x
\right)\\
&\geq
F\!\left(
\alpha D^{2}\varphi^{+}_{1},
\alpha D\varphi_1^{+},
\log^{b}\!\left(\frac{B}{\varphi_1^{+}}\right),
x
\right)
+\mathcal{P}^{-}_{\lambda,\Lambda}
\!\left(-\beta D\varphi^{+}_{1}\otimes D\varphi^{+}_{1}\right)\\
&=
F\!\left(
\alpha D^{2}\varphi^{+}_{1},
\alpha D\varphi_1^{+},
\alpha\varphi_1^{+},
x
\right)
+
\Big[
F\!\left(
\alpha D^{2}\varphi^{+}_{1},
\alpha D\varphi_1^{+},
\log^{b}\!\left(\frac{B}{\varphi_1^{+}}\right),
x
\right)\\
&\quad-
F\!\left(
\alpha D^{2}\varphi^{+}_{1},
\alpha D\varphi_1^{+},
\alpha\varphi_1^{+},
x
\right)
\Big]
+\beta\lambda |\nabla\varphi_1^{+}|^{2}\\
&\geq
F\!\left(
\alpha D^{2}\varphi^{+}_{1},
\alpha D\varphi_1^{+},
\alpha\varphi_1^{+},
x
\right)
-\gamma\Big(
\log^{b}\!\left(\frac{B}{\varphi_1^{+}}\right)
-\alpha\varphi_1^{+}
\Big)^{+}
+\beta\lambda |\nabla\varphi_1^{+}|^{2}\\
&=
\alpha\mu_1^{+}\varphi_1^{+}
-\gamma\Big(
\log^{b}\!\left(\frac{B}{\varphi_1^{+}}\right)
-\alpha\varphi_1^{+}
\Big)^{+}
+\beta\lambda |\nabla\varphi_1^{+}|^{2}\\
&=
\alpha\mu_1^{+}\varphi_1^{+}
-\gamma\Big(
\log^{b}\!\left(\frac{B}{\varphi_1^{+}}\right)
-\alpha\varphi_1^{+}
\Big)
+\beta\lambda |\nabla\varphi_1^{+}|^{2}
\end{aligned}
\right.
\end{equation*}
substituting the expressions of $\alpha$ and $\beta$ in the above inequality,
\begin{equation*}
\left\{
\begin{aligned}
F(D^2w,Dw,w,x)
&\ge
\left[-\frac{b}{\varphi_1^{+}}
\log^{b-1}\!\left(\frac{B}{\varphi_1^{+}}\right)\right]
\mu_1^{+}\varphi_1^{+}
-\gamma\!\left[
\log^{b}\!\left(\frac{B}{\varphi_1^{+}}\right)
+b\log^{b-1}\!\left(\frac{B}{\varphi_1^{+}}\right)
\right]\\
&\quad
+\left[
-\frac{b}{(\varphi_1^{+})^{2}}
\log^{b-1}\!\left(\frac{B}{\varphi_1^{+}}\right)
\left\{
\frac{(b-1)+\log\!\left(\frac{B}{\varphi_1^{+}}\right)}
{\log\!\left(\frac{B}{\varphi_1^{+}}\right)}
\right\}
\right]\lambda |\nabla\varphi_1^{+}|^2\\[1mm]
&=
\frac{1}{(\varphi_1^{+})^{2}}
\log^{b-1}\!\left(\frac{B}{\varphi_1^{+}}\right)
\Bigg[
-b\mu_1^{+}(\varphi_1^{+})^{2}
-\gamma(\varphi_1^{+})^{2}
\Big(
\log\!\left(\frac{B}{\varphi_1^{+}}\right)+b
\Big)\\
&\qquad
-b\left\{
\frac{(b-1)+\log\!\left(\frac{B}{\varphi_1^{+}}\right)}
{\log\!\left(\frac{B}{\varphi_1^{+}}\right)}
\right\}
\lambda |\nabla\varphi_1^{+}|^2
\Bigg].
\end{aligned}
\right.
\end{equation*}
In the neighbourhood $N$ of  $\partial\Omega$, the quantity
\[
\begin{aligned}
&\Bigg[
-b\mu_1^+(\varphi_1^+)^2
-\gamma(\varphi_1^+)^2\left\{
\log\!\left( \frac{B}{\varphi_1^{+}}\right)+b
\right\}
-b\left\{
\frac{(b-1)+\log\!\left(\frac{B}{\varphi_1^{+}}\right)}
{\log\!\left(\frac{B}{\varphi_1^{+}}\right)}
\right\}
\lambda|\nabla\varphi_1^+|^2
\Bigg]
\end{aligned}
\]
is positive. Therefore, there exists a constant $C_1>0$ such that
\[
F(D^2w,Dw,w,x)\geq 
C_1\,\frac{1}{(\varphi_1^+)^2}
\log^{\,b-1}\!\left( \frac{B}{\varphi_1^{+}}\right)
\qquad \text{in } \Omega\cap N,
\]
Notice that also that for any compact set $\mathcal{K}\subset\Omega$, the quantity
\[
\begin{aligned}
&\Bigg[
-b\mu_1^+(\varphi_1^+)^2
-\gamma(\varphi_1^+)^2\left\{
\log\!\left( \frac{B}{\varphi_1^{+}}\right)+b
\right\}
-b\left\{
\frac{(b-1)+\log\!\left(\frac{B}{\varphi_1^{+}}\right)}
{\log\!\left(\frac{B}{\varphi_1^{+}}\right)}
\right\}
\lambda|\nabla\varphi_1^+|^2
\Bigg]
\end{aligned}
\]
is strictly positive in $\mathcal{K}$
if
\[
\gamma<
\frac{
-b\mu_1^+(\varphi_1^+)^2
-
b\left\{
\frac{(b-1)+\log\!\left(\frac{B}{\varphi_1^{+}(x)}\right)}
{\log\!\left(\frac{B}{\varphi_1^{+}(x)}\right)}
\right\}
\lambda|\nabla\varphi_1^+|^2
}{
(\varphi_1^+)^2
\left\{
\log\!\left( \frac{B}{\varphi_1^{+}}\right)+b
\right\}
}.
\]
Observe that above condition can always be obtained by choosing $B$ large. Furthermore, there exists a constant $C_1>0$ such that
\[
F(D^2w,Dw,w,x)
\geq
C_1\,(\varphi_1^+)^{-2}
\log^{\,b-1}\!\left( \frac{B}{\varphi_1^{+}}\right),
\qquad \text{in }~\mathcal{K}.
\]
Combining this estimate with the upper bound obtained previously, we conclude that there exist
constants $C_1,C_2>0$ such that
\[
C_1\,(\varphi_1^+)^{-2}
\log^{\,b-1}\!\left( \frac{B}{\varphi_1^{+}}\right)
\leq
F(D^2w,Dw,w,x)
\leq
C_2\,(\varphi_1^+)^{-2}
\log^{\,b-1}\!\left( \frac{B}{\varphi_1^{+}}\right),
\qquad \text{in } \Omega.
\]
Equivalently, recalling that $w=\log^{\,b}\!\left(\dfrac{B}{\varphi_1^{+}}\right)$, we may write
\[
C_1\,(\varphi_1^+)^{-2}
\log^{-a}\!\left( \frac{B}{\varphi_1^{+}}\right)
w^{-p}
\leq
F(D^2w,Dw,w,x)
\leq
C_2\,(\varphi_1^+)^{-2}
\log^{-a}\!\left( \frac{B}{\varphi_1^{+}}\right)
w^{-p},
\qquad \text{in } \Omega.
\]

From \eqref{ineigen}, it follows that the functions
\[
\underline{u}=m w
\quad\text{and}\quad
\overline{u}=M w
\]
are, respectively, a subsolution and a supersolution of \eqref{exeq6}, provided that
$m>0$ is sufficiently small and $M>1$ is sufficiently large.\\
\textbf{Uniqueness:} Observe that for each fixed $x\in\Omega$ the function 
\begin{equation}
t\in(0,\infty)\rightarrow \delta(x)^{-2}\log^{-a}\left(\frac{A}{\delta(x)}\right)t^{-p}
\end{equation}
is decreasing in $t.$ Thus, the operator considered here is proper and uniqueness follows. More precisely, 
let $u$ and $v$ be two distinct solutions of \ref{exeq6}. So without loss of generality we can assume that $\omega=\{x\in \Omega~|~v(x)<u(x)\}$ is non-empty open set and also $\omega\subset\subset \Omega.$ Thus in $\omega,$
\begin{equation*}
\left\{
\begin{aligned}
F(D^2u(x),Du(x),u(x),x)&= \delta(x)^{-2}\log^{-a}\left(\frac{A}{\delta(x)}\right)u^{-p}\\
&\leq\delta(x)^{-2}\log^{-a}\left(\frac{A}{\delta(x)}\right)v^{-p} \\
&=F(D^2v(x),Dv(x),v(x),x),
\end{aligned}
\right.
\end{equation*}   
and $v=u$ on $\partial\omega.$ Thus $u\leq v$ in $\omega,$ which contradicts the definition of $\omega.$ Thus $u\equiv v.$
\end{proof}
\begin{corollary}\label{coro3}
Let $C>0$, $p\geq 0$, $A>\operatorname{diam}(\Omega)$, and $a>1$.
Then there exists a constant $c>0$ such that any solution
$u$ of
\begin{equation}\label{ineq-coro}
\left\{
\begin{aligned}
F(D^2u,Du,u,x)
&\geq
\delta(x)^{-2}
\log^{-a}\!\left(\frac{A}{\delta(x)}\right)
u^{-p}
\quad &&\text{in } \Omega,\\
u&>0
\quad &&\text{in } \Omega,\\
u&=0
\quad &&\text{on } \partial\Omega,
\end{aligned}
\right.
\end{equation}
satisfies the boundary lower estimate
\[
u(x)\geq
c\,\log^{\frac{1-a}{1+p}}\!\left(\frac{A}{\delta(x)}\right),
\qquad \text{for all } x\in\Omega.
\]
\end{corollary}
\section{Main Results}

This section is devoted to the study of the existence and non-existence of positive solutions to problem \eqref{main}. 
The analysis is divided into several cases, depending on the relationships among the exponents $p$, $q$, $r$, and $s$.

Throughout this section, we repeatedly use the qualitative properties of solutions to \eqref{main} established in \eqref{prop3}. 
In addition, depending on the interplay between the exponents, Proposition~\ref{prop4} will play a crucial role in the construction of suitable sub- and supersolutions.

We begin by establishing non-existence results for \eqref{main}. 
These results rely on a comparison argument and on the non-existence criterion provided in Proposition~\ref{th1}.
\begin{theorem}[Non-existence]\label{nonexis1}
Let $p,s\geq 0$ and $q,r>0$ be such that at least one of the following conditions holds:
\begin{enumerate}
    \item $r\,\displaystyle \min\left\{1,\frac{2-q}{1+p}\right\}\geq 2$;
    
    \item $q\,\displaystyle \min\left\{1,\frac{2-r}{1+s}\right\}\geq 2$;
    
    \item $p>\max\{1,r-1\}$, \quad $2r>(1-s)(1+p)$, \quad and \quad $q(1+p-r)>(1+p)(1+s)$;
    
    \item $s>\max\{1,q-1\}$, \quad $2q>(1-p)(1+s)$, \quad and \quad $r(1+s-q)>(1+p)(1+s)$.
\end{enumerate}
Then the system \eqref{main} admits no positive solutions.
\end{theorem}
\begin{proof}
Since the system \eqref{main} is invariant under the transformation
\[
(u,v,p,q,r,s)\longmapsto (v,u,s,r,q,p),
\]
it suffices to prove cases \textbf{(1)} and \textbf{(3)}.

\medskip
\noindent\textbf{Case (1).}
Observe that the hypothesis in this case implies $0<q<2$. We further divide the proof into the following three subcases:
\[
p+q<1,\qquad p+q=1,\qquad p+q>1.
\]

Assume that $(u,v)$ is a solution of the system \eqref{main}. By Proposition \eqref{prop3}, there exists $c>0$ such that the estimate \eqref{exeq} holds.

\medskip
\noindent\textbf{Subcase $p+q<1$.}
From hypothesis \textbf{(1)} we deduce that $r\geq 2$. Using the estimate \eqref{exeq} in the first equation of the system \eqref{main}, we obtain
\begin{equation}\label{neq22}
\left\{
\begin{aligned}
F(D^2u,Du,u,x) &\leq c_{1}\,\delta(x)^{-q}u^{-p}
    && \text{in }\Omega,\\
u &>0 && \text{in }\Omega,\\
u &=0 && \text{on }\partial\Omega.
\end{aligned}
\right.
\end{equation}
for some $c_{1}>0$. By Proposition \eqref{prop4}(i), we deduce that
\[
u(x)\leq c_{2}\,\delta(x)\quad \text{in }\Omega,
\]
for some constant $c_{2}>0$. Using this estimate in the second equation of \eqref{main}, we obtain
\begin{equation}
\left\{
\begin{aligned}
F(D^2v,Dv,v,x) &\geq c_{3}\,\delta(x)^{-r}v^{-s}
    && \text{in }\Omega,\\
v &>0 && \text{in }\Omega,\\
v &=0 && \text{on }\partial\Omega,
\end{aligned}
\right.
\end{equation}
where $c_{3}>0$. By Corollary \eqref{coro}, this is impossible since $r\geq 2$.

\medskip
\noindent\textbf{Subcase $p+q=1$.}
From condition \textbf{(1)}, it follows that $r\geq 2$. As in the previous cases, we observe that $u$ satisfies \eqref{neq22} for some constant $c_{1}>0$. By Proposition \eqref{prop4}(ii), there exists a constant $c_{6}>0$ such that
\[
u(x)\leq c_{6}\,\delta(x)\,
\log^{\frac{1}{1+p}}\!\left(\frac{A}{\delta(x)}\right)
\quad \text{in }\Omega,
\]
for some $A>3\,\mathrm{diam}(\Omega)$. Substituting this estimate into the second equation of \eqref{main}, we obtain
\begin{equation}\label{neq24}
\left\{
\begin{aligned}
F(D^2v,Dv,v,x) 
&\geq c_{7}\,\delta(x)^{-r}
\log^{-\frac{r}{1+p}}\!\left(\frac{A}{\delta(x)}\right)v^{-s}
\quad \text{in }\Omega,\\
v&>0 \quad \text{in }\Omega,\\
v&=0 \quad \text{on }\partial\Omega,
\end{aligned}
\right.
\end{equation}
where $c_{7}$ is a positive constant. From Theorem \eqref{th1} it follows that
\[\int_{0}^{1} t^{1-r} \log ^{-\frac{r}{1+p}}\left(\frac{A}{t}\right) d t<\infty
\]
Since $r\geq 2$, the above integrability condition forces $r=2$. 
Now, applying \eqref{neq24} with $r=2$ and Corollary \eqref{coro3}, 
there exists a constant $c_{8}>0$ such that
\begin{equation}\label{neq25}
v(x)\geq c_{8}\,
\log^{\frac{p-1}{(1+p)(1+s)}}\!\left(\frac{A}{\delta(x)}\right)
\quad \text{in }\Omega .
\end{equation}

Substituting the estimate \eqref{neq25} into the first equation of the system \eqref{main},
we obtain
\begin{equation}\label{neq26}
\left\{
\begin{aligned}
F(D^2u,Du,u,x) 
&\leq c_{9}\,
\log^{\frac{q(1-p)}{(1+p)(1+s)}}\!\left(\frac{A}{\delta(x)}\right)
u^{-p}
\quad \text{in }\Omega,\\
u&>0 \quad \text{in }\Omega,\\
u&=0 \quad \text{on }\partial\Omega,
\end{aligned}
\right.
\end{equation}
for some constant $c_{9}>0$. Fix $0<a<1-p$. Then, by \eqref{neq26}, 
there exists a constant $c_{10}>0$ such that $u$ satisfies
\begin{equation}
\left\{
\begin{aligned}
F(D^2u,Du,u,x) 
&\leq c_{10}\,\delta(x)^{-a}u^{-p}
\quad \text{in }\Omega,\\
u&>0 \quad \text{in }\Omega,\\
u&=0 \quad \text{on }\partial\Omega.
\end{aligned}
\right.
\end{equation}
By Proposition \eqref{prop4}(i) (since $a+p<1$), we obtain
\[
u(x)\leq c_{11}\,\delta(x)\qquad \text{in }\Omega,
\]
for some constant $c_{11}>0$. Using this estimate in the second equation of
\eqref{main}, and recalling that $r=2$, we finally obtain
\begin{equation}
\left\{
\begin{aligned}
F(D^2v,Dv,v,x)
&\geq c_{12}\,\delta(x)^{-2}v^{-s}
\quad \text{in }\Omega,\\
v&>0 \quad \text{in }\Omega,\\
v&=0 \quad \text{on }\partial\Omega,
\end{aligned}
\right.
\end{equation}
which is impossible according to Corollary \eqref{coro}. Therefore, the system \eqref{main} has no solutions.\\
\medskip
\noindent\textbf{Subcase $p+q>1$.}
From hypothesis {\rm(1)} we also have
\[
\frac{r(2-q)}{1+p}\geq 2.
\]
As in the previous cases, $u$ satisfies \eqref{neq22}. Hence, by Proposition
\eqref{prop4}(iii), there exists a constant $c_{4}>0$ such that
\[
u(x)\leq c_{4}\,\delta(x)^{\frac{2-q}{1+p}}
\qquad \text{in }\Omega.
\]

Using this estimate in the second equation of system \eqref{main}, we obtain
\begin{equation}
\left\{
\begin{aligned}
F(D^2v,Dv,v,x)
&\geq c_{5}\,\delta(x)^{-\frac{r(2-q)}{1+p}}\,v^{-s}
\quad \text{in }\Omega,\\
v&>0 \quad \text{in }\Omega,\\
v&=0 \quad \text{on }\partial\Omega,
\end{aligned}
\right.
\end{equation}
for some $c_{5}>0$, which is impossible in view of Corollary \eqref{coro},
since
\[
\frac{r(2-q)}{1+p}\geq 2.
\]
\textbf{Case(3).} Suppose that the system \eqref{main} admits a solution $(u,v)$ and set
\[
M=\max_{x\in\overline{\Omega}} v(x).
\]
From the first equation of \eqref{main}, we obtain
\begin{equation*}
\left\{
\begin{aligned}
F(D^2u,Du,u,x) &\geq c_{1}u^{-p} \quad \text{in } \Omega,\\
u &>0 \quad \text{in } \Omega,\\
u &=0 \quad \text{on } \partial\Omega,
\end{aligned}
\right.
\end{equation*}
where $c_{1}=M^{-q}>0$. By Proposition \eqref{prop4}(iii), there exists $c_{2}>0$ such that
\[
u(x)\geq c_{2}\,\delta(x)^{\frac{2}{1+p}} \quad \text{in } \Omega.
\]
Combining this estimate with the second equation of \eqref{main}, we obtain
\begin{equation*}
\left\{
\begin{aligned}
F(D^2v,Dv,v,x) &\leq c_{3}\,\delta(x)^{-\frac{2r}{1+p}}\,v^{-s} \quad \text{in } \Omega,\\
v &>0 \quad \text{in } \Omega,\\
v &=0 \quad \text{on } \partial\Omega,
\end{aligned}
\right.
\end{equation*}
for some $c_{3}>0$. Since
\[
\frac{2r}{1+p}+s>1,
\]
another application of Proposition \eqref{prop4}(iii) yields the estimate
\[
v(x)\leq c_{4}\,\delta(x)^{\frac{2(1+p-r)}{(1+p)(1+s)}} \quad \text{in } \Omega,
\]
for some $c_{4}>0$. Substituting this bound into the first equation of \eqref{main}, we find that there exists $c_{5}>0$ such that
\begin{equation*}
\left\{
\begin{aligned}
F(D^2u,Du,u,x) &\geq c_{5}\,\delta(x)^{-\frac{2q(1+p-r)}{(1+p)(1+s)}}\,u^{-p}
\quad \text{in } \Omega,\\
u &>0 \quad \text{in } \Omega,\\
u &=0 \quad \text{on } \partial\Omega.
\end{aligned}
\right.
\end{equation*}
This contradicts Corollary \eqref{coro}, since
\[
q(1+p-r)>(1+p)(1+s).
\]
Hence, the system \eqref{main} admits no solutions.
\end{proof}
In order to state our next theorem related to existence of solution to \eqref{main}, we need the following notations:
\[
\alpha := (p+q)\min\left\{1,\frac{2-r}{1+s}\right\},
\qquad
\beta := (r+s)\min\left\{1,\frac{2-q}{1+p}\right\}.
\]
As we have mentioned in the introduction, the existence of solutions will be established by employing Schauder fixed theorem in an appropriate cone. These cones have been defined with the help of the boundary behaviour of solution. These boundary behaviour are consequence of Proposition \ref{prop4} and H\"{o}pf lemma. Further more in order to apply the Schauder fixed point theorem we need to prove that the solution operator is compact which will be obtained as a consequence of results of Section 4\cite{felmer2012existence}.

\begin{theorem}[Existence]\label{th13}
Let $p,s\geq0$ and $q,r>0$ be such that $(1+p)(1+s)-qr>0$, and assume one of the following conditions holds:
\begin{enumerate}
    \item $\alpha \leq 1$ and $r < 2$;
    \item $\beta \leq 1$ and $q < 2$;
    \item $p,s \geq 1$ and $q,r < 2$.
\end{enumerate}
Then the system \eqref{main} has at least one positive solution.
\end{theorem}
\begin{proof}
As in the proof of non-existence, cases $1$ and $2$ are symmetric. Hence, we will provide the proof for case $1$ and case $3$.  

The proof of case $1$ is further divided into six subcases, depending on the inequalities satisfied by the exponents $p, q, r$ and $s$. More precisely, we consider:  
\begin{itemize}
    \item[(I)] $r+s>1$ and $\alpha = p + \frac{q(2-r)}{1+s} < 1$,
    \item[(II)] $r+s=1$ and $\alpha = p + q < 1$,
    \item[(III)] $r+s<1$ and $\alpha = p + q < 1$,
    \item[(IV)] $r+s<1$ and $\alpha = p + q = 1$,
    \item[(V)] $r+s>1$ and $\alpha = p + q \frac{2-r}{1+s} = 1$,
    \item[(VI)] $r+s=1$ and $\alpha = p + q = 1$.
\end{itemize}
\textbf{Subcase (I):} Consider the case 
\[
r+s>1 \quad \text{and} \quad \alpha = p + \frac{q(2-r)}{1+s} < 1.
\]

Since $\alpha < 1$, we can apply Proposition \ref{prop4}(i), which guarantees the existence of constants $0 < c_1 < 1 < c_2$ such that any subsolution $\underline{u}$ (resp. supersolution $\overline{u}$) of the problem 
\begin{equation}\label{eq8}
\left\{
\begin{aligned}
F(D^2u,Du,u,x) &= \delta(x)^{-\frac{q(2-r)}{1+s}} u^{-p} &&\text{in } \Omega,\\
u &= 0 &&\text{on } \partial\Omega,
\end{aligned}
\right.
\end{equation}
satisfies 
\begin{equation}\label{eq9}
\underline{u}(x) \leq c_2 \delta(x) \quad (\text{resp. } c_1 \delta(x) \leq \overline{u}(x)) \quad \text{for all } x \in \Omega.
\end{equation}
Since $r+s>1$, Proposition \ref{prop4}(iii) implies that any subsolution $\underline{v}$ (resp. supersolution $\overline{v}$) of 
\begin{equation}\label{eq10}
\left\{
\begin{aligned}
F(D^2v,Dv,v,x) &= \delta(x)^{-r} v^{-s} && \text{in } \Omega,\\
v &= 0 && \text{on } \partial\Omega,
\end{aligned}
\right.
\end{equation}
satisfies
\begin{equation}\label{eq11}
\underline{v}(x) \leq c_2 \delta(x)^{\frac{2-r}{1+s}} \quad (\text{resp. } \overline{v}(x) \geq c_1 \delta(x)^{\frac{2-r}{1+s}}) \quad \text{for all } x \in \Omega.
\end{equation}
Taking into account that $(1+p)(1+s)-qr>0$, we can choose constants $0 < m_i < 1 < M_i$ for $i=1,2$ such that
\begin{equation}\label{eq81}
\left\{
\begin{aligned}
&(i)\quad M_1^{\frac{r}{1+s}} m_2 \leq c_1 < c_2 \leq M_1 m_2^{\frac{q}{1+p}},\\[2mm]
&(ii)\quad M_2^{\frac{q}{1+p}} m_1 \leq c_1 < c_2 \leq M_2 m_1^{\frac{r}{1+s}}.
\end{aligned}
\right.
\end{equation}
Set
\[
\mathcal{A} = \Bigl\{ (u,v) \in C(\overline{\Omega}) \times C(\overline{\Omega}) :
m_1 \delta(x) \leq u(x) \leq M_1 \delta(x), \;\;
m_2 \delta(x)^{\frac{2-r}{1+s}} \leq v(x) \leq M_2 \delta(x)^{\frac{2-r}{1+s}} \;\;
\text{in } \Omega \Bigr\}.
\]
\textbf{Definition of $T$:} For any $(u,v)\in\mathcal{A}$ we define $(Tu,Tv)$ is the unique viscosity solution of the "de-coupled" system: 
    \begin{equation}\label{eq14}
\left\{
\begin{aligned}
F(D^2(Tu),D(Tu),Tu,x) &= v^{-q}(Tu)^{-p} &&\text{in } \Omega,\\
F(D^2(Tv),D(Tv),Tv,x) &= u^{-r}(Tv)^{-s} &&\text{in } \Omega,\\
Tu = Tv &= 0 &&\text{on } \partial\Omega,
\end{aligned}
\right.
\end{equation}
Before, we proceed, we prove the existence and uniqueness of solution of \eqref{eq14}. As the system of equations is decoupled, it is sufficient to prove the existence and uniqueness solution of each equation separately. We consider the second equation, the result concerning for the first equation follows in the same way.\\
\textbf{Existence of Solution: } Consider second equation from \eqref{eq14} and denote $w=Tv;$
\begin{equation}\label{Tv}
\left\{
\begin{aligned}
F(D^2w,Dw,w,x) &= u^{-r}w^{-s} &&\text{in } \Omega,\\
w &= 0 &&\text{on } \partial\Omega,
\end{aligned}
\right.
\end{equation}
As $(u,v)\in\mathcal{A}$ so the from definition of $\mathcal{A},$ we have:
\[m_1\delta(x)\leq u(x)\leq M_1\delta(x).\]
Using this inequality in \eqref{Tv}, we get
\begin{equation}\label{auxi}
M_1^{-r}\delta(x)^{-r}w^{-s}\leq F(D^2w,Dw,w,x)\leq m_1^{-r}\delta(x)^{-r}w^{-s}.
\end{equation}
Furthermore, by noting that $r+s>1,$ we can construct sub and supersolution of \eqref{Tv} as in the case Proposition \ref{prop4}(iii). Thus the existence of solution of solution followed by applying Theorem 4.1\cite{crandall1992user}.
Noting that the uniqueness of the solution follows from Proposition~\ref{prop1}.\\
Thus the map $\mathcal{F}:\mathcal{A}\longrightarrow C(\Bar{\Omega})\times C(\Bar{\Omega})$: 
\[\mathcal{F}(u,v)=(Tu,Tv)~~~\text{for any}~~~ (u,v)\in\mathcal{A},\]
is well defined.
Our aim is to show that $\mathcal{F}$ has a fixed point in $\mathcal{A}$. This will be achieved by employing Schauder's Fixed Point Theorem once we establish the following:  
\begin{enumerate}
    \item $\mathcal{F}(\mathcal{A}) \subseteq \mathcal{A}$,
    \item $\mathcal{F}$ is completely continuous; that is, $\mathcal{F}$ is both compact and continuous.
\end{enumerate}
\textbf{Proof of $\mathcal{F}(\mathcal{A})\subseteq\mathcal{A}$:} \\

Let $(u,v) \in \mathcal{A}$. Then 
\[
v(x) \leq M_2 \delta(x)^{\frac{2-r}{1+s}} \quad \text{in } \Omega.
\]

From \eqref{eq14}, $Tu$ satisfies
\begin{equation*}
\left\{
\begin{aligned}
F(D^2(Tu),D(Tu),Tu,x) &\geq M_2^{-q} \delta(x)^{\frac{-q(2-r)}{1+s}} (Tu)^{-p} && \text{in } \Omega, \\
Tu &= 0 && \text{on } \partial\Omega.
\end{aligned}
\right.
\end{equation*}

Thus, $\overline{u} := M_2^{\frac{q}{1+p}} Tu$ is a supersolution of \eqref{eq8}. By \eqref{eq81} and \eqref{eq9}, we have
\[
Tu = M_2^{\frac{-q}{1+p}} \overline{u} \geq c_1 M_2^{\frac{-q}{1+p}} \delta(x) \geq m_1 \delta(x) \quad \text{in } \Omega.
\]

Also, note that for $(u,v) \in \mathcal{A}$, 
\[
v(x) \geq m_2 \delta(x)^{\frac{2-r}{1+s}} \quad \text{in } \Omega.
\]

From \eqref{eq14}, $Tu$ satisfies
\begin{equation*}
\left\{
\begin{aligned}
F(D^2(Tu),D(Tu),Tu,x) &\leq m_2^{-q} \delta(x)^{\frac{-q(2-r)}{1+s}} (Tu)^{-p} && \text{in } \Omega, \\
Tu &= 0 && \text{on } \partial\Omega.
\end{aligned}
\right.
\end{equation*}

Thus, $\underline{u} := m_2^{\frac{q}{1+p}} Tu$ is a subsolution of \eqref{eq8}. By \eqref{eq81} and \eqref{eq9}, we obtain
\[
Tu = m_2^{\frac{-q}{1+p}} \underline{u} \leq m_2^{\frac{-q}{1+p}} c_2 \delta(x) \leq M_1 \delta(x) \quad \text{in } \Omega.
\]

Hence,
\[
m_1 \delta(x) \leq Tu \leq M_1 \delta(x) \quad \text{in } \Omega.
\]

Similarly, one can show
\[
m_2 \delta(x)^{\frac{2-r}{1+s}} \leq Tv \leq M_2 \delta(x)^{\frac{2-r}{1+s}} \quad \text{in } \Omega.
\]

Therefore,
\[
\mathcal{F}(\mathcal{A}) \subseteq \mathcal{A}.
\]
\textbf{Proof that $\mathcal{F}$ is completely continuous:} \\

\textbf{Compactness:} From the above observation, $\mathcal{F}(u,v) \in \mathcal{A}$ for any $(u,v) \in \mathcal{A}$. Thus, we have 
\[
\|Tu\|_{L^{\infty}(\Omega)}, \, \|Tv\|_{L^{\infty}(\Omega)} \leq C.
\]

Furthermore, in view of the definition of $\mathcal{A}$, we can find $a>0$ such that
\begin{equation*}
\left\{
\begin{aligned}
0 \leq F(D^2(Tu), D(Tu), Tu, x) &\leq c \, \delta(x)^{-a} \quad \text{in } \Omega, \\
0 \leq F(D^2(Tv), D(Tv), Tv, x) &\leq c \, \delta(x)^{-a} \quad \text{in } \Omega.
\end{aligned}
\right.
\end{equation*}

By using Theorem 9 in \cite{felmer2012existence}, it follows that
\[
Tu, Tv \in {C}^{0,\gamma}(\overline{\Omega}), \quad 0<\gamma<1.
\]

Since ${C}^{0,\gamma}(\overline{\Omega})$ is compactly embedded in ${C}(\overline{\Omega})$, we conclude that $\mathcal{F}$ is compact. This argument will be used repeatedly in the subsequent steps. \\

\textbf{Continuity:} Let $(u_n, v_n) \subset \mathcal{A}$ be such that 
\[
u_n \longrightarrow u \quad \text{and} \quad v_n \longrightarrow v \quad \text{in } {C}(\overline{\Omega}) \quad \text{as } n \longrightarrow \infty.
\]

Using the fact that $\mathcal{F}$ is compact, there exists $(U,V) \in \mathcal{A}$ such that, up to a subsequence,
\[
Tu_n \longrightarrow U, \quad Tv_n \longrightarrow V \quad \text{in } {C}(\overline{\Omega}) \quad \text{as } n \longrightarrow \infty.
\]
By using the stability result for viscosity solutions [Section 6\cite{crandall1992user}], we find that 
\begin{equation*}
\left\{
\begin{aligned}
F(D^2U, DU, U, x) &= v^{-q} U^{-p} &&\text{in } \Omega, \\
F(D^2V, DV, V, x) &= u^{-r} V^{-s} &&\text{in } \Omega, \\
U = V &= 0 &&\text{on } \partial\Omega.
\end{aligned}
\right.
\end{equation*}

By uniqueness of \eqref{eq14}, it follows that $Tu = U$ and $Tv = V$. Hence,
\[
Tu_n \longrightarrow Tu, \quad Tv_n \longrightarrow Tv \quad \text{in } {C}(\overline{\Omega}) \quad \text{as } n \longrightarrow \infty.
\]

Therefore, $\mathcal{F}$ is continuous. By Schauder's Fixed Point Theorem, there exists $(u,v) \in \mathcal{A}$ such that 
\[
\mathcal{F}(u,v) = (u,v).
\]

Consequently, $(u,v)$ is a solution of system \ref{main}.\\
\textbf{Subcase (II):} $r+s=1$ and $\alpha=p+q<1.$

Using Proposition \ref{prop4}(i)-(ii), there exist constants $0 < c_1 < 1 < c_2$ and $0 < a < 1$ such that the following hold:

Any subsolution $\underline{u}$ of the problem
\begin{equation}\label{exeq50}
\left\{
\begin{aligned}
F(D^2 u, Du, u, x) &= \delta(x)^{-q} u^{-p} && \text{in } \Omega,\\
u &= 0 && \text{on } \partial \Omega,
\end{aligned}
\right.
\end{equation}
satisfies
\begin{equation}\label{exeq51}
\underline{u}(x) \leq c_2 \, \delta(x) \quad \text{for } x \in \Omega,
\end{equation}

and any supersolution $\overline{u}$ of the problem
\begin{equation}\label{exeq52}
\left\{
\begin{aligned}
F(D^2 u, Du, u, x) &= \delta(x)^{-q(1-a)} u^{-p} && \text{in } \Omega,\\
u &= 0 && \text{on } \partial \Omega,
\end{aligned}
\right.
\end{equation}
satisfies
\begin{equation}\label{exeq53}
\overline{u}(x) \geq c_1 \, \delta(x) \quad \text{for } x \in \Omega.
\end{equation}

Similarly, any subsolution $\underline{v}$ (resp. supersolution $\overline{v}$) of the problem
\begin{equation}\label{exeq54}
\left\{
\begin{aligned}
F(D^2 v, Dv, v, x) &= \delta(x)^{-r} v^{-s} && \text{in } \Omega,\\
v &= 0 && \text{on } \partial \Omega,
\end{aligned}
\right.
\end{equation}
satisfies
\begin{equation}\label{exeq55}
\underline{v}(x) \leq c_2 \, \delta(x) \log^{\frac{1}{1+s}}\left(\frac{A}{\delta(x)}\right), \quad
\Big(\text{resp. } c_1 \, \delta(x) \log^{\frac{1}{1+s}}\left(\frac{A}{\delta(x)}\right) \leq \overline{v}(x)\Big), \quad x \in \Omega.
\end{equation}
We can take $A$ large enough such that 
\[
\log^{\frac{1}{1+s}}\left(\frac{A}{\delta(x)}\right) \geq 1,
\] 
and consequently, we have
\begin{equation}\label{exeq56}
\overline{v}(x) \geq c_1 \, \delta(x).
\end{equation}

If we show that 
\[
\log^{\frac{1}{1+s}}\left(\frac{A}{\delta(x)}\right) \leq \delta(x)^{-a} \quad \text{for some } a,
\] 
then it follows that
\begin{equation}\label{exeq57}
\underline{v}(x) \leq c_2 \, \delta(x)^{1-a}.
\end{equation}

To see this, consider the function 
\[
f(x) = x^a - \log(x).
\] 
By elementary calculus, $f$ attains its minimum at 
\[
x_0 = \left(\frac{1}{a}\right)^{\frac{1}{a}},
\] 
with 
\[
f(x_0) = \frac{1}{a} (1 + \log(a)) \geq 0 \quad \text{for } a > \frac{1}{e}.
\] 

Therefore, there exists $a \in \left(\frac{1}{e}, 1\right)$ such that 
\[
\log^{\frac{1}{1+s}}\left(\frac{A}{\delta(x)}\right) \leq \delta(x)^{-a},
\] 
and hence,
\[
\underline{v}(x) \leq c_2 \, \delta(x)^{1-a}.
\]
 We now define
 $$\mathcal{A}=\left\{(u,v)\in C(\overline{\Omega})\times C(\overline{\Omega}): m_1\delta(x)\leq u(x)\leq M_1\delta(x)~~\text{and}~~m_2\delta(x)\leq v(x)\leq M_2\delta(x)^{1-a}~~\text{in}~~\Omega\right\}$$
where $0<m_i<1<M_i~(i=1,2)$ satisfy \eqref{eq81} and 
 \begin{equation}\label{exeq35}
    m_2\big[\text{diam}(\Omega)\big]^a<M_2 
 \end{equation}
  For any $(u,v)\in\mathcal{A},$ we consider $(Tu,Tv)$ the unique solution of the decoupled system 
    \begin{equation}\label{c2eq14}
\left\{
\begin{aligned}
F(D^2(Tu),D(Tu),Tu,x) &= v^{-q}(Tu)^{-p} &&\text{in } \Omega,\\
F(D^2(Tv),D(Tv),Tv,x) &= u^{-r}(Tv)^{-s} &&\text{in } \Omega,\\
Tu = Tv &= 0 &&\text{on } \partial\Omega,
\end{aligned}
\right.
\end{equation}
Define 
\[
\mathcal{F}:\mathcal{A} \longrightarrow C(\overline{\Omega}) \times C(\overline{\Omega}) \quad \text{by} \quad \mathcal{F}(u,v) = (Tu,Tv) \quad \text{for any } (u,v) \in \mathcal{A}.
\]

Our aim is to show that $\mathcal{F}$ has a fixed point in $\mathcal{A}$. This will be achieved by employing Schauder's Fixed Point Theorem once we establish:

\begin{enumerate}
    \item $\mathcal{F}(\mathcal{A}) \subseteq \mathcal{A}$,
    \item $\mathcal{F}$ is completely continuous, that is, compact and continuous.
\end{enumerate}

\textbf{Step 1:} $\mathcal{F}(\mathcal{A}) \subseteq \mathcal{A}.$\\

Let $(u,v) \in \mathcal{A}.$ Then 
\[
v(x) \leq M_2 \, \delta(x)^{1-a} \quad \text{in } \Omega.
\]

From \eqref{c2eq14}, $Tu$ satisfies
\begin{equation*}
\left\{
\begin{aligned}
F(D^2(Tu), D(Tu), Tu, x) &\geq M_2^{-q} \, \delta(x)^{-q(1-a)} (Tu)^{-p} && \text{in } \Omega,\\
Tu &= 0 && \text{on } \partial \Omega.
\end{aligned}
\right.
\end{equation*}
Thus, 
\[
\overline{u} := M_2^{\frac{q}{1+p}} Tu
\]
is a supersolution of \eqref{exeq52}. By \eqref{exeq53} and \eqref{eq81}, we obtain
\[
Tu = M_2^{\frac{-q}{1+p}} \overline{u} \geq c_1 M_2^{\frac{-q}{1+p}} \delta(x) \geq m_1 \delta(x) \quad \text{in } \Omega.
\]

Similarly, from $v(x) \geq m_2 \delta(x) \quad \text{in } \Omega$ and the definition of $Tu$, we deduce that
\begin{equation*}
\left\{
\begin{aligned}
F(D^2(Tu), D(Tu), Tu, x) &\leq m_2^{-q} \delta(x)^{-q} (Tu)^{-p} && \text{in } \Omega,\\
Tu &= 0 && \text{on } \partial \Omega.
\end{aligned}
\right.
\end{equation*}
Thus, 
\[
\underline{u} := m_2^{\frac{q}{1+p}} Tu
\]
is a subsolution of problem \eqref{exeq50}. Hence, from \eqref{exeq51} and \eqref{eq81}, we obtain
\[
Tu = m_2^{\frac{-q}{1+p}} \underline{u} \leq c_2 m_2^{\frac{-q}{1+p}} \delta(x) \leq M_1 \delta(x) \quad \text{in } \Omega.
\]

Therefore, we have proved that $Tu$ satisfies
\[
m_1 \delta(x) \leq Tu \leq M_1 \delta(x) \quad \text{in } \Omega.
\]

In a similar manner, we can construct 
\[
\overline{v} := M_1^{\frac{r}{1+s}} Tv
\]
as a supersolution of \eqref{exeq54}. Using \eqref{exeq56} and \eqref{eq81}, we obtain
\[
Tv = M_1^{\frac{-r}{1+s}} \overline{v} \geq c_1 \delta(x) M_1^{\frac{-r}{1+s}} \geq m_2 \delta(x) \quad \text{in } \Omega.
\]
and 
\[
\underline{v} := m_1^{\frac{r}{1+s}} Tv
\]
is a subsolution of the problem \eqref{exeq54}. By using \eqref{exeq57} and \eqref{eq81}, we obtain
\[
Tv = m_1^{\frac{-r}{1+s}} \underline{v} \leq m_1^{\frac{-r}{1+s}} c_2 \delta(x)^{1-a}.
\]

Thus, we have proved that
\[
m_2 \delta(x) \leq Tv \leq M_2 \delta(x)^{1-a} \quad \text{in } \Omega.
\]

\textbf{Step 2:} $\mathcal{F}$ \emph{is compact and continuous}, can be proved as in Case (I).\\

\textbf{Subcase (III):} $r+s<1$ and $\alpha=p+q<1$:\\
By Proposition \ref{prop4} (i), there exist constants $0<c_1<1<c_2$ such that any subsolution $\underline{u}$ (resp. supersolution $\overline{u}$) of the problem
\begin{equation}
\left\{
\begin{aligned}
F(D^2 u, Du, u, x) &= \delta(x)^{-q} u^{-p} && \text{in } \Omega,\\
u &= 0 && \text{on } \partial \Omega,
\end{aligned}
\right.
\end{equation}
satisfies
\begin{equation}
\underline{u}(x) \leq c_2 \delta(x) \quad (\text{resp. } c_1 \delta(x) \leq \overline{u}(x)) \quad \text{for } x \in \Omega.
\end{equation}

Similarly, by Proposition \ref{prop4} (i), there exist constants $0<c_1<1<c_2$ such that any subsolution $\underline{v}$ (resp. supersolution $\overline{v}$) of the problem
\begin{equation}
\left\{
\begin{aligned}
F(D^2 v, Dv, v, x) &= \delta(x)^{-r} v^{-s} && \text{in } \Omega,\\
v &= 0 && \text{on } \partial \Omega,
\end{aligned}
\right.
\end{equation}
satisfies
\begin{equation}
\underline{v}(x) \leq c_2 \delta(x) \quad (\text{resp. } c_1 \delta(x) \leq \overline{v}(x)) \quad \text{for } x \in \Omega.
\end{equation}

Set
\[
\mathcal{A} = \Bigl\{ (u,v) \in C(\overline{\Omega}) \times C(\overline{\Omega}) : m_1 \delta(x) \leq u(x) \leq M_1 \delta(x) \text{ and } m_2 \delta(x) \leq v(x) \leq M_2 \delta(x) \text{ in } \Omega \Bigr\},
\]
where $0 < m_i < 1 < M_i$ $(i=1,2)$ satisfy \eqref{eq81}.
The remaining part of this subcase can be derived in a similar way as in above subcases (I) and (II).\\
\textbf{Subcase (IV):} $r+s<1$ and $\alpha=p+q=1.$ Interchange $u$ and $v$ in the initial system \eqref{main} and proceed as in subcase (II).\\
\textbf{Subcase (V):} $r+s>1$ and $\alpha = p + q\left(\frac{2-r}{1+s}\right) = 1.$  
Fix $0<a<1$ such that $ar+s>1.$  

From Proposition~\ref{prop4}\,(i), (iii), there exist constants $0<c_1<1<c_2$ such that the following hold.

Any subsolution $\underline{u}$ of the problem
\begin{equation*}
\left\{
\begin{aligned}
F(D^2u,Du,u,x) &= \delta(x)^{-\frac{q(2-ar)}{1+s}}u^{-p}, \quad u>0 && \text{in } \Omega,\\
u &= 0 && \text{on } \partial\Omega,
\end{aligned}
\right.
\end{equation*}
satisfies
\begin{equation*}
    \underline{u}(x) \leq c_2\,\delta(x)^a \quad \text{in } \Omega.
\end{equation*}

Similarly, any supersolution $\overline{u}$ of the problem
\begin{equation*}
\left\{
\begin{aligned}
F(D^2u,Du,u,x) &= \delta(x)^{-\frac{q(2-r)}{1+s}}u^{-p}, \quad u>0 && \text{in } \Omega,\\
u &= 0 && \text{on } \partial\Omega,
\end{aligned}
\right.
\end{equation*}
satisfies
\begin{equation*}
    \overline{u}(x) \geq c_1\,\delta(x) \quad \text{in } \Omega.
\end{equation*}

Next, any subsolution $\underline{v}$ of problem~\eqref{eq10} satisfies
\begin{equation*}
    \underline{v}(x) \leq c_2\,\delta(x)^{\frac{2-r}{1+s}} \quad \text{in } \Omega.
\end{equation*}

Moreover, any supersolution $\overline{v}$ of the problem
\begin{equation*}
\left\{
\begin{aligned}
F(D^2v,Dv,v,x) &= \delta(x)^{-ar}v^{-s}, \quad v>0 && \text{in } \Omega,\\
v &= 0 && \text{on } \partial\Omega,
\end{aligned}
\right.
\end{equation*}
satisfies
\begin{equation*}
    \overline{v}(x) \geq c_1\,\delta(x)^{\frac{2-ar}{1+s}} \quad \text{in } \Omega.
\end{equation*}

We now define the set
\begin{align*}
\mathcal{A} = \bigg\{(u,v)\in C(\overline{\Omega}) \times C(\overline{\Omega}) :\;&
m_1\,\delta(x) \leq u(x) \leq M_1\,\delta(x)^a,\\
& m_2\,\delta(x)^{\frac{2-ar}{1+s}} \leq v(x) \leq M_2\,\delta(x)^{\frac{2-r}{1+s}}
\quad \text{in } \Omega \bigg\},
\end{align*}
where $0<m_i<1<M_i$ $(i=1,2)$ satisfy condition~\eqref{eq81}, with constants $c_1$ and $c_2$ as defined above, and
\[
m_1\,\big(\mathrm{diam}(\Omega)\big)^{1-a} < M_1,
\qquad
m_2\,\big(\mathrm{diam}(\Omega)\big)^{\frac{r(1-a)}{1+s}} < M_2.
\]
Remaining part of this can be proved in a similar way as in above subcases (I) and (II).\\

\textbf{Subcase (VI):} $r+s=1$ and $\alpha = p+q = 1.$  
We proceed in the same manner by considering the set
\[
\mathcal{A} = \left\{(u,v)\in C(\overline{\Omega}) \times C(\overline{\Omega}) :
\begin{aligned}
& m_1\,\delta(x) \leq u(x) \leq M_1\,\delta(x)^{1-a},\\
& m_2\,\delta(x) \leq v(x) \leq M_2\,\delta(x)^{1-a}
\end{aligned}
\quad \text{in } \Omega \right\},
\]
where $0<a<1$ is a fixed constant, and the constants $m_i, M_i$ $(i=1,2)$ satisfy condition~\eqref{eq81} for suitable constants $c_1, c_2>0.$ Moreover, we assume that
\[
m_i\,\big(\mathrm{diam}(\Omega)\big)^a < M_i, \qquad i=1,2.
\]


\textbf{Case (3).} Let 
\[
a = \frac{2(1+s-q)}{(1+p)(1+s)-qr}, \qquad b = \frac{2(1+p-r)}{(1+p)(1+s)-qr}.
\] 
Then we have
\[
(1+p)a + bq = 2, \qquad ar + (1+s)b = 2.
\]

Since $p+bq>1$ and $s+ar>1$, by Proposition \ref{prop4}(iii) we can find constants $0<c_1<1<c_2$ such that:  

Any subsolution $\underline{u}$ and any supersolution $\overline{u}$ of the problem
\begin{equation}
\left\{
\begin{aligned}
F(D^2 u, Du, u, x) &= \delta(x)^{-bq} u^{-p} && \text{in } \Omega,\\
u &= 0 && \text{on } \partial \Omega,
\end{aligned}
\right.
\end{equation}
satisfy
\[
\overline{u}(x) \geq c_1 \delta(x)^a, \qquad \underline{u}(x) \leq c_2 \delta(x)^a \quad \text{in } \Omega.
\]

Any subsolution $\underline{v}$ and any supersolution $\overline{v}$ of the problem
\begin{equation}
\left\{
\begin{aligned}
F(D^2 v, Dv, v, x) &= \delta(x)^{-ar} v^{-s} && \text{in } \Omega,\\
v &= 0 && \text{on } \partial \Omega,
\end{aligned}
\right.
\end{equation}
satisfy
\[
\overline{v}(x) \geq c_1 \delta(x)^b, \qquad \underline{v}(x) \leq c_2 \delta(x)^b \quad \text{in } \Omega.
\]

We now define
\[
\mathcal{A} = \Bigl\{ (u,v) \in C(\overline{\Omega}) \times C(\overline{\Omega}) : m_1 \delta(x)^a \leq u(x) \leq M_1 \delta(x)^a,~~ m_2 \delta(x)^b \leq v(x) \leq M_2 \delta(x)^b \text{ in } \Omega \Bigr\},
\]
where $0 < m_i < 1 < M_i$ $(i=1,2)$ satisfy \eqref{eq81}.
Remaining part of this can be proved in a similar way as in above subcases (I) and (II).\\

\end{proof}
\begin{theorem}[$C^1$-regularity]
Let $p, s \geq 0$ and $q, r > 0$ satisfy $(1+p)(1+s) - qr > 0$. Introduce
\[
\alpha = p + q \, \min\left\{1, \frac{2-r}{1+s} \right\}, \qquad 
\beta = r + s \, \min\left\{1, \frac{2-q}{1+p} \right\}.
\]
Then the following statements hold:
\begin{enumerate}
    \item System \ref{main} has a solution $(u,v)$ with $u \in C^1(\overline{\Omega})$ if and only if $\alpha < 1$ and $r < 2$.
    \item System \ref{main} has a solution $(u,v)$ with $v \in C^1(\overline{\Omega})$ if and only if $\beta < 1$ and $q < 2$.
    \item System \ref{main} has a solution $(u,v)$ with $u, v \in C^1(\overline{\Omega})$ if and only if $p + q < 1$ and $r + s < 1$.
\end{enumerate}
\end{theorem}
\begin{proof}
\textbf{(1).} Assume first that the system \ref{main} has a solution $(u,v)$ with $u \in C^1(\overline{\Omega})$.  
Then there exists $c>0$ such that 
\[
u(x) \leq c \, \delta(x) \quad \text{in } \Omega.
\]  
Using this estimate in the second equation of system \eqref{main}, we deduce that $v$ satisfies the elliptic inequality 
\begin{equation*}
\left\{
\begin{aligned}
F(D^2v, Dv, v, x) &\geq c_3 \delta(x)^{-r} v^{-s} &&\text{in } \Omega,\\
v &> 0 &&\text{in } \Omega,\\
v &= 0 &&\text{on } \partial \Omega,
\end{aligned}
\right.
\end{equation*}
for some constant $c_3 > 0$. By Theorem \ref{th1}, this implies $r < 2$.  

Now, we prove that $\alpha < 1$. Suppose, by contradiction, that $\alpha \geq 1$. We divide the argument into three cases.

\textbf{Case (i):} $r+s > 1$. Then 
\[
\alpha = p + \frac{q(2-r)}{1+s} \geq 1.
\]  
From Proposition \ref{prop3}, we have $u(x) \geq c \, \delta(x)$ in $\Omega$ for some $c>0$.  
Hence, $v$ satisfies
\begin{equation*}
\left\{
\begin{aligned}
F(D^2v, Dv, v, x) &\leq c_1 \delta(x)^{-r} v^{-s} &&\text{in } \Omega,\\
v &> 0 &&\text{in } \Omega,\\
v &= 0 &&\text{on } \partial \Omega,
\end{aligned}
\right.
\end{equation*}
where $c_1>0$. Since $r < 2$.  
 Since $r<2$, by Proposition \ref{prop4}(iii) we obtain
\[
v(x)\leq c_2\,\delta(x)^{\frac{2-r}{1+s}} \quad \text{in } \Omega,
\]
for some $c_2>0$. Using this estimate in the first equation of system \eqref{main}, we deduce that $u$ satisfies
\begin{equation}\label{eq37}
\left\{
\begin{aligned}
F(D^2u,Du,u,x) &\geq c_3\,\delta(x)^{-\frac{q(2-r)}{1+s}}u^{-p} &&\text{in } \Omega,\\
u &>0 &&\text{in } \Omega,\\
u &=0 &&\text{on } \partial\Omega,
\end{aligned}
\right.
\end{equation}
where $c_3>0$.  

If $\frac{q(2-r)}{1+s}\geq 2$, then by Corollary \ref{coro} the above inequality admits no solution, which is a contradiction. Hence, we may assume that
\[
\frac{q(2-r)}{1+s}<2.
\]

If $\alpha>1$, then by Proposition \ref{prop4}(iii) there exists $c_4>0$ such that
\[
u(x)\geq c_4\,\delta(x)^{\tau} \quad \text{in } \Omega,
\]
where
\[
\tau=\frac{2-\frac{q(2-r)}{1+s}}{1+p}\in (0,1).
\]

Recall that there exists $\varphi_1^+$ satisfying
\[
c\,\delta(x)\leq \varphi_1^+(x)\leq \delta(x) \quad \text{in } \Omega,
\]
for some $0<c<1$. Consequently,
\begin{equation}
u(x)\geq c_4\,\delta(x)^{\tau}\geq c_4\,\big(\varphi_1^+(x)\big)^{\tau}
\quad \text{in } \Omega.
\end{equation}

Fix a point $x_0\in \partial\Omega$, and let $n$ denote the outward unit normal vector to $\partial\Omega$ at $x_0$.
\begin{equation*}
\left\{
\begin{aligned}
    \frac{\partial u}{\partial n}(x_0)=\lim_{t \nearrow 0}\frac{u(x_0+tn)-u(x_0)}{t}
    &= \lim_{t \nearrow 0}\frac{u(x_0+tn)}{t}\\
    &\leq c_4\lim_{t \nearrow 0}\frac{\varphi_1^+(x_0+tn)-\varphi_1^+(x_0)}{t}(\varphi_1^+)^{\tau-1}(x_0+tn)\\
    &=c_4\frac{\partial\varphi_1^+}{\partial n}(x_0)\lim_{t \nearrow 0}(\varphi_1^+)^{\tau-1}(x_0+tn)\\
    &=-\infty. ~~~~~~~~~~~~~~~~~~\left (\text{as}~\frac{\partial\varphi_1^+}{\partial n}(x_0)<0\right)
\end{aligned}
\right.
\end{equation*}
Hence, $u\notin C^1(\overline{\Omega})$, which contradicts the hypothesis.

If $\alpha=1$, then from \eqref{eq37} and Proposition \ref{prop4}(ii) we obtain
\[
u(x)\geq c_5\,\delta(x)\log^{\frac{1}{1+p}}\!\left(\frac{A}{\delta(x)}\right)
\quad \text{in } \Omega,
\]
for some $c_5>0$.  

Recall that
\[
c\,\delta(x)\leq \varphi_1^+(x)\leq \delta(x) \quad \text{in } \Omega,
\]
for some $0<c<1$. Therefore,
\begin{equation}
u(x)\geq c_5\,\delta(x)\log^{\frac{1}{1+p}}\!\left(\frac{A}{\delta(x)}\right)
\geq c_6\,\varphi_1^+(x)\log^{\frac{1}{1+p}}\!\left(\frac{A}{\varphi_1^+(x)}\right),
\end{equation}
for some constant $c_6>0$.
Fix $x_0 \in \partial\Omega$, and let $n$ be the outer unit normal vector to $\partial\Omega$ at $x_0$.\\
\begin{equation*}
\left\{
\begin{aligned}
    \frac{\partial u}{\partial n}(x_0)
    &= \lim_{t \nearrow 0} \frac{u(x_0 + t n) - u(x_0)}{t} \\
    &= \lim_{t \nearrow 0} \frac{u(x_0 + t n)}{t} \\
    &\le c_6 \lim_{t \nearrow 0}
    \frac{\varphi_1^+(x_0 + t n) - \varphi_1^+}{t}
    \log^{\frac{1}{1+p}}\!\left( \frac{A}{\varphi_1^+(x_0 + t n)} \right) \\
    &= c_6 \frac{\partial \varphi_1^+}{\partial n}(x_0)
    \lim_{t \nearrow 0}
    \log^{\frac{1}{1+p}}\!\left( \frac{A}{\varphi_1^+(x_0 + t n)} \right) \\
    &= -\infty
    \qquad
    \left( \text{since } \frac{\partial \varphi_1^+}{\partial n}(x_0) < 0 \right).
\end{aligned}
\right.
\end{equation*}
Hence, $u\notin C^1(\overline{\Omega}).$ which contradicts the hypothesis.\\
\textbf{Case (ii):} $r+s<1$. Then $\alpha = p+q \ge 1$.
From Proposition~\ref{prop3}, we have $u(x) \ge c\,\delta(x)$ in $\Omega$, for some $c>0$.
Thus, $v$ satisfies
\begin{equation}
\left\{
\begin{aligned}
F(D^2 v, D v, v, x) &\le c_1 \delta(x)^{-r} v^{-s} \quad \text{in } \Omega,\\
v &> 0 \quad \text{in } \Omega,\\
v &= 0 \quad \text{on } \partial\Omega.
\end{aligned}
\right.
\end{equation}
By Proposition~\ref{prop4}(i), we find that $v(x) \le c_7 \delta(x)$ in $\Omega$, for some $c_7>0$.
Thus, $u$ satisfies
\begin{equation}
\left\{
\begin{aligned}
F(D^2 u, D u, u, x) &\ge c_8 \delta(x)^{-q} u^{-p} \quad \text{in } \Omega,\\
u &> 0 \quad \text{in } \Omega,\\
u &= 0 \quad \text{on } \partial\Omega.
\end{aligned}
\right.
\end{equation}
where $c_8>0$. From Corollary~\ref{coro}, it follows that $q<2$.\\
If $\alpha>1$, then by Proposition~\ref{prop4}(iii) we have
$u(x) \ge c_9 \delta(x)^{\tau}$ in $\Omega$, where
$\tau = \frac{2-q}{1+p} \in (0,1)$ and $c_9>0$.\\
Note that
\[
c\,\delta(x) \le \varphi_1^+(x) \le \delta(x)
\quad \text{for some } c>0.
\]
Then
\begin{equation}
u(x) \ge c_9 \delta(x)^{\tau}
\ge c_9 \bigl(\varphi_1^+(x)\bigr)^{\tau}
\quad \text{in } \Omega.
\end{equation}
Fix $x_0 \in \partial\Omega$, and let $n$ be the outer unit normal vector to
$\partial\Omega$ at $x_0$.\\
\begin{equation*}
\left\{
\begin{aligned}
\frac{\partial u}{\partial n}(x_0)
&= \lim_{t \nearrow 0} \frac{u(x_0 + t n) - u(x_0)}{t} \\
&= \lim_{t \nearrow 0} \frac{u(x_0 + t n)}{t} \\
&\le c_9 \lim_{t \nearrow 0}
\frac{\varphi_1^+(x_0 + t n) - \varphi_1^+(x_0)}{t}
\bigl(\varphi_1^+\bigr)^{\tau-1}(x_0 + t n) \\
&= c_9 \frac{\partial \varphi_1^+}{\partial n}(x_0)
\lim_{t \nearrow 0}
\bigl(\varphi_1^+\bigr)^{\tau-1}(x_0 + t n) \\
&= -\infty,
\qquad
\left( \text{since } \frac{\partial \varphi_1^+}{\partial n}(x_0) < 0 \right).
\end{aligned}
\right.
\end{equation*}
Hence, $u \notin C^1(\overline{\Omega})$, which contradicts the hypothesis.\\
If $\alpha = 1$, then we can carry out a similar calculation as in Case~(i) above.\\
\textbf{Case (iii):} $r+s=1$. Then $\alpha = p+q \ge 1$.
From Proposition~\ref{prop3}, we have $u(x) \ge c\,\delta(x)$ in $\Omega$, for some $c>0$.
Therefore, $v$ satisfies
\begin{equation}
\left\{
\begin{aligned}
F(D^2 v, D v, v, x) &\le c_1 \delta(x)^{-r} v^{-s} \quad \text{in } \Omega,\\
v &> 0 \quad \text{in } \Omega,\\
v &= 0 \quad \text{on } \partial\Omega.
\end{aligned}
\right.
\end{equation}
By Proposition~\ref{prop4}(ii), we obtain
\[
v(x) \le c_{10} \delta(x)
\log^{\frac{1}{1+s}}\!\left( \frac{A}{\delta(x)} \right)
\quad \text{in } \Omega.
\]
where $c_{10}>0$. It follows that $u$ satisfies
\begin{equation}\label{eq44}
\left\{
\begin{aligned}
F(D^2 u, D u, u, x)
&\ge c^{\star}_{10} \delta(x)^{-q}
\log^{\frac{-q}{1+s}}\!\left( \frac{A}{\delta(x)} \right)
u^{-p}
\quad \text{in } \Omega,\\
u &> 0 \quad \text{in } \Omega,\\
u &= 0 \quad \text{on } \partial\Omega.
\end{aligned}
\right.
\end{equation}
where $c^{\star}_{10}>0.$ From Corollary \ref{coro}, we have $q-b<2,$  where $b=\frac{q}{1+s}$.\\
If $\alpha=p+q>1$, we fix $0<b<\text{min}\{q,p+q-1\}$ and from \eqref{eq44} we have that $u$ satisfies 
\begin{equation}
\left\{
\begin{aligned}
F(D^2u,Du,u,x)& \geq c_{11}\delta(x)^{-(q-b)}u^{-p}~\text{in}~\Omega,\\
u&>0~\text{in}~\Omega,\\
u&=0~\text{on} ~\partial\Omega,
\end{aligned}
\right.
\end{equation}
for some $c_{11}>0$. Now, since $p+q-b>1$, from Proposition~\ref{prop4}(iii) we find that
\[
u(x) \ge c_{12} \delta(x)^{\frac{2-(q-b)}{1+p}}
\quad \text{in } \Omega,
\]
where $c_{12}>0$. Since
\[
0 < \frac{2-(q-b)}{1+p} < 1,
\]
proceeding as in the previous cases, we find that
\[
\frac{\partial u}{\partial n} = -\infty,
\]
which contradicts the hypothesis.\\
If $\alpha = p+q = 1$, that is, $p+q = r+s = 1$. First, if $q < 1+s$, that is,
$q \neq 1$ and $s \neq 0$, by \eqref{eq44} and Corollary~\ref{coro2} we deduce
\[
u(x) \ge c_{13} \delta(x)
\log^{\frac{1+s-q}{(1+p)(1+s)}}\!\left( \frac{A}{\delta(x)} \right)
\quad \text{in } \Omega,
\]
for some $c_{13}>0$. Proceeding as before, we obtain
\[
\frac{\partial u}{\partial n} = -\infty
\quad \text{on } \partial\Omega,
\]
which is again impossible.\\
If $q = 1$ and $s = 0$, then we apply Proposition~\ref{prop11} to obtain
\[
u(x) \ge c_{14} \delta(x)
\log\!\left[ \log\!\left( \frac{A}{\delta(x)} \right) \right]
\quad \text{in } \Omega,
\]
where $c_{14}>0$. This also leads to the same contradiction,
\[
\frac{\partial u}{\partial n} = -\infty
\quad \text{on } \partial\Omega.
\]
Thus, we have proved that if the system~(1) has a solution $(u, v)$ with
$u \in C^{1}(\overline{\Omega})$, then $\alpha < 1$ and $r < 2$.\\
Conversely, assume now that $\alpha < 1$ and $r < 2$. By Theorem~\ref{th13}(1)
(Cases~I, II, and~III), there exists a solution $(u, v)$ of~\eqref{main} such that
\[
u(x) \ge c\,\delta(x) \quad \text{in } \Omega,
\]
and
\[
v(x) \ge c\,\delta(x)
\quad \text{in } \Omega, \quad \text{if } r+s \le 1,
\]
or
\[
v(x) \ge c\,\delta(x)^{\frac{2-r}{1+s}}
\quad \text{in } \Omega, \quad \text{if } r+s > 1,
\]
for some $c>0$. Using the above estimates, we find that
\[
F(D^2 u, D u, u, x) = u^{-p} v^{-q}
\le C\,\delta(x)^{-\alpha}
\quad \text{in } \Omega,
\]
for some $C>0$. Now we can use the interior $C^{1,\alpha}_{\mathrm{loc}}(\Omega)$
regularity to conclude that
\[
u \in C^{1,\alpha}_{\mathrm{loc}}(\Omega)
\]
for some $\alpha$. Moreover, by adapting the proofs of Propositions~2 and~3~\cite{felmer2012existence},
we obtain Hölder continuity of the gradient of the solution near the boundary.
Thus, by combining the interior and boundary regularity, we find that
\[
u \in C^{1,1-\alpha}(\overline{\Omega}).
\]
\textbf{(3).} Assume that the system~\eqref{main} has a solution $(u, v)$ with
$u, v \in C^1(\overline{\Omega})$. Then there exists $c>0$ such that
$v(x) \le c\,\delta(x)$ in $\Omega$. Using this estimate in the first equation
of~\eqref{main}, we find that
\begin{equation*}
\left\{
\begin{aligned}
F(D^2 u, D u, u, x) &\ge C\,\delta(x)^{-q} u^{-p}
\quad \text{in } \Omega,\\
u &> 0 \quad \text{in } \Omega,\\
u &= 0 \quad \text{on } \partial\Omega,
\end{aligned}
\right.
\end{equation*}
where $C$ is a positive constant. From Corollary~\eqref{coro}, we have $q<2$.\\
If $p+q>1$, then by Proposition~\eqref{prop4}(iii) we obtain
$u(x) \ge c_1 \delta(x)^{\tau}$, where
$\tau = \frac{2-q}{1+p} \in (0,1)$ and $c_1>0$.\\
Note that
\[
c\,\delta(x) \le \varphi_1^+(x) \le \delta(x)
\quad \text{for some } c>0.
\]
Then
\begin{equation}
u(x) \ge c_1 \delta(x)^{\tau}
\ge c_1 \bigl(\varphi_1^+(x)\bigr)^{\tau}
\quad \text{in } \Omega.
\end{equation}
Fix $x_0 \in \partial\Omega$, and let $n$ be the outer unit normal vector to
$\partial\Omega$ at $x_0$.\\
\begin{equation*}
\left\{
\begin{aligned}
\frac{\partial u}{\partial n}(x_0)
&= \lim_{t \nearrow 0} \frac{u(x_0 + t n) - u(x_0)}{t} \\
&= \lim_{t \nearrow 0} \frac{u(x_0 + t n)}{t} \\
&\le c_1 \lim_{t \nearrow 0}
\frac{\varphi_1^+(x_0 + t n) - \varphi_1^+(x_0)}{t}
\bigl(\varphi_1^+\bigr)^{\tau-1}(x_0 + t n) \\
&= c_1 \frac{\partial \varphi_1^+}{\partial n}(x_0)
\lim_{t \nearrow 0} \bigl(\varphi_1^+\bigr)^{\tau-1}(x_0 + t n) \\
&= -\infty,
\qquad
\left( \text{since } \frac{\partial \varphi_1^+}{\partial n}(x_0) < 0 \right).
\end{aligned}
\right.
\end{equation*}
Hence, $u \notin C^1(\overline{\Omega})$, which contradicts the hypothesis.\\
If $p+q = 1$, then by Proposition~\eqref{prop4}(ii) we have
\[
u(x) \ge c_2 \delta(x) \log^{\frac{1}{1+p}}\!\left( \frac{A}{\delta(x)} \right), \quad c_2 > 0.
\]
Note that
\[
c \,\delta(x) \le \varphi_1^+(x) \le \delta(x)
\quad \text{for some } c>0.
\]
Then
\begin{equation}
u(x) \ge c_2 \delta(x) \log^{\frac{1}{1+p}}\!\left( \frac{A}{\delta(x)} \right)
\ge c_3 \varphi_1^+(x) \log^{\frac{1}{1+p}}\!\left( \frac{A}{\varphi_1^+(x)} \right)
\quad \text{in } \Omega,
\end{equation}
where $c_3 > 0$.  

Fix $x_0 \in \partial\Omega$, and let $n$ be the outer unit normal vector to
$\partial\Omega$ at $x_0$.\\
\begin{equation*}
\left\{
\begin{aligned}
\frac{\partial u}{\partial n}(x_0)
&= \lim_{t \nearrow 0} \frac{u(x_0 + t n) - u(x_0)}{t} \\
&= \lim_{t \nearrow 0} \frac{u(x_0 + t n)}{t} \\
&\le c_3 \lim_{t \nearrow 0}
\frac{\varphi_1^+(x_0 + t n) - \varphi_1^+(x_0)}{t}
\log^{\frac{1}{1+p}}\!\left( \frac{A}{\varphi_1^+(x_0 + t n)} \right) \\
&= c_3 \frac{\partial \varphi_1^+}{\partial n}(x_0)
\lim_{t \nearrow 0} \log^{\frac{1}{1+p}}\!\left( \frac{A}{\varphi_1^+(x_0 + t n)} \right) \\
&= -\infty,
\qquad
\left( \text{since } \frac{\partial \varphi_1^+}{\partial n}(x_0) < 0 \right).
\end{aligned}
\right.
\end{equation*}
Hence, $u \notin C^1(\overline{\Omega})$, which contradicts the hypothesis.\\
Thus, $p+q < 1$, and in a similar way we can obtain $r+s < 1$.\\
Assume now that $p+q < 1$ and $r+s < 1$. Then, by the Existence Theorem~\eqref{th13}(i) (Case~III), we have that~\eqref{main} has a solution $(u,v)$ such that
\[
u(x), v(x) \ge c\,\delta(x) \quad \text{in } \Omega, \quad \text{for some } c>0.
\]
This yields
\begin{equation*}
\left\{
\begin{aligned}
F(D^2 u, D u, u, x) &\le C\,\delta(x)^{-(p+q)} \quad \text{in } \Omega,\\
F(D^2 v, D v, v, x) &\le C\,\delta(x)^{-(r+s)} \quad \text{in } \Omega,
\end{aligned}
\right.
\end{equation*}
where $C>0$. Note that $(p+q),(r+s) < 1$, so again, using the interior regularity and boundary regularity as in Propositions~2 and~3~\cite{felmer2012existence}, we conclude that
\[
u, v \in C^1(\overline{\Omega}).
\]
\end{proof}
\begin{theorem}[Uniqueness]
Let $p,s \ge 0$ and $q,r > 0$ satisfy $(1+p)(1+s) - q r > 0$, and assume one of the following conditions holds:
\begin{enumerate}
    \item $p+q < 1$ and $r < 2$;
    \item $r+s < 1$ and $q < 2$.
\end{enumerate}
Then, the system~\eqref{main} has a unique solution.
\end{theorem}
\begin{proof}
\textbf{Case (1):} $p+q < 1$ and $r < 2$.\\
Let $(u_1, v_1)$ and $(u_2, v_2)$ be two solutions of the system~\eqref{main}. Then, by Proposition~\ref{prop3}, there exists $c_1 > 0$ such that
\begin{equation}\label{eq15}
u_i, v_i \ge c_1 \delta(x) \quad \text{in } \Omega, \quad i=1,2.
\end{equation}

Hence, $u_i$ satisfies
\begin{equation*}
\left\{
\begin{aligned}
F(D^2 u_i, D u_i, u_i, x) &\le c_2 \delta(x)^{-q} u_i^{-p} \quad \text{in } \Omega,\\
u_i &> 0 \quad \text{in } \Omega,\\
u_i &= 0 \quad \text{on } \partial\Omega,
\end{aligned}
\right.
\end{equation*}
for some $c_2 > 0$. By Proposition~\ref{prop4}(i) and~\eqref{eq15}, there exists $0 < c < 1$ such that
\begin{equation}\label{eq16}
c \delta(x) \le u_i(x) \le \frac{1}{c} \delta(x) \quad \text{in } \Omega, \quad i=1,2.
\end{equation}
This implies that we can find a constant $C>1$ such that $C u_1 \ge u_2$ and $C u_2 \ge u_1$ in $\Omega$.\\

We claim that $u_1 \ge u_2$ in $\Omega$. Suppose, on the contrary, that
\[
M = \inf \{ A > 1 : A u_1 \ge u_2 \text{ in } \Omega \}.
\]
By our assumption, $M > 1$. From $M u_1 \ge u_2$ in $\Omega$, it follows that
\[
F(D^2 v_2, D v_2, v_2, x) = u_2^{-r} v_2^{-s} \ge M^{-r} u_1^{-r} v_2^{-s} \quad \text{in } \Omega.
\]
Therefore, $v_1$ is a solution and $M^{\frac{r}{1+s}} v_2$ is a supersolution of
\begin{equation*}
\left\{
\begin{aligned}
F(D^2 w, D w, w, x) &= u_1^{-r} w^{-s} \quad \text{in } \Omega,\\
w &> 0 \quad \text{in } \Omega,\\
w &= 0 \quad \text{on } \partial\Omega.
\end{aligned}
\right.
\end{equation*}
By Proposition~\ref{prop1}, we obtain
\[
v_1 \le M^{\frac{r}{1+s}} v_2 \quad \text{in } \Omega.
\]
The above estimate yields
\[
F(D^2 u_1, D u_1, u_1, x) = u_1^{-p} v_1^{-q} \ge M^{-\frac{qr}{1+s}} u_1^{-p} v_2^{-q} \quad \text{in } \Omega.
\]
It follows that $u_2$ is a solution and $M^{\frac{qr}{(1+s)(1+p)}} u_1$ is a supersolution of
\begin{equation*}
\left\{
\begin{aligned}
F(D^2 w, D w, w, x) &= v_2^{-q} w^{-p} \quad \text{in } \Omega,\\
w &> 0 \quad \text{in } \Omega,\\
w &= 0 \quad \text{on } \partial\Omega.
\end{aligned}
\right.
\end{equation*}
By Proposition~\ref{prop1}, we obtain
\[
M^{\frac{qr}{(1+s)(1+p)}} u_1 \ge u_2 \quad \text{in } \Omega.
\]
Since $M > 1$ and $\frac{qr}{(1+s)(1+p)} < 1$, the above inequality contradicts the minimality of $M$. Hence, $u_1 \ge u_2$ in $\Omega$. Similarly, we deduce $u_1 \le u_2$ in $\Omega$. Therefore, $u_1 \equiv u_2$, which also yields $v_1 \equiv v_2$ (using the comparison principle). Thus, the system has a unique solution.\\

\textbf{Case (2):} $r+s < 1$ and $q < 2$.\\
Let $(u_1, v_1)$ and $(u_2, v_2)$ be two solutions of the system~\eqref{main}. Then, by Proposition~\ref{prop3}, there exists $c_1 > 0$ such that
\begin{equation}\label{c2eq15}
u_i, v_i \ge c_1 \delta(x) \quad \text{in } \Omega, \quad i=1,2.
\end{equation}

Hence, $v_i$ satisfies
\begin{equation*}
\left\{
\begin{aligned}
F(D^2 v_i, D v_i, v_i, x) &\le c_2 \delta(x)^{-r} v_i^{-s} \quad \text{in } \Omega,\\
v_i &> 0 \quad \text{in } \Omega,\\
v_i &= 0 \quad \text{on } \partial\Omega,
\end{aligned}
\right.
\end{equation*}
for some $c_2 > 0$. By Proposition~\ref{prop4}(i) and~\eqref{c2eq15}, there exists $0 < c < 1$ such that
\begin{equation}\label{c2eq16}
c \delta(x) \le v_i(x) \le \frac{1}{c} \delta(x) \quad \text{in } \Omega, \quad i=1,2.
\end{equation}
This implies that we can find a constant $C > 1$ such that $C v_1 \ge v_2$ and $C v_2 \ge v_1$ in $\Omega$.\\

We claim that $v_1 \ge v_2$ in $\Omega$. Suppose, on the contrary, that
\[
M = \inf \{ A > 1 : A v_1 \ge v_2 \text{ in } \Omega \}.
\]
By our assumption, $M > 1$. From $M v_1 \ge v_2$ in $\Omega$, it follows that
\[
F(D^2 u_2, D u_2, u_2, x) = u_2^{-p} v_2^{-q} \ge M^{-q} u_2^{-p} v_1^{-q} \quad \text{in } \Omega.
\]
Therefore, $u_1$ is a solution and $M^{\frac{q}{1+p}} u_2$ is a supersolution of
\begin{equation*}
\left\{
\begin{aligned}
F(D^2 w, D w, w, x) &= v_1^{-q} w^{-p} \quad \text{in } \Omega,\\
w &> 0 \quad \text{in } \Omega,\\
w &= 0 \quad \text{on } \partial\Omega.
\end{aligned}
\right.
\end{equation*}
By Proposition~\ref{prop1}, we obtain
\[
u_1 \le M^{\frac{q}{1+p}} u_2 \quad \text{in } \Omega.
\]
The above estimate yields
\[
F(D^2 v_1, D v_1, v_1, x) = u_1^{-r} v_1^{-s} \ge M^{-\frac{q r}{1+p}} u_2^{-r} v_1^{-s} \quad \text{in } \Omega.
\]
It follows that $v_2$ is a solution and $M^{\frac{q r}{(1+s)(1+p)}} v_1$ is a supersolution of
\begin{equation*}
\left\{
\begin{aligned}
F(D^2 w, D w, w, x) &= u_2^{-r} w^{-s} \quad \text{in } \Omega,\\
w &> 0 \quad \text{in } \Omega,\\
w &= 0 \quad \text{on } \partial\Omega.
\end{aligned}
\right.
\end{equation*}
By Proposition~\ref{prop1}, we obtain
\[
M^{\frac{q r}{(1+s)(1+p)}} v_1 \ge v_2 \quad \text{in } \Omega.
\]
Since $M > 1$ and $\frac{q r}{(1+s)(1+p)} < 1$, the above inequality contradicts the minimality of $M$. Hence, $v_1 \ge v_2$ in $\Omega$. Similarly, we deduce $v_1 \le v_2$ in $\Omega$. Therefore, $v_1 \equiv v_2$, which also yields $u_1 \equiv u_2$ (using the comparison principle). Thus, the system has a unique solution.

\end{proof}
\section{Acknowledgement}
Mohan Mallick is happy to acknowledge the support ANRF/ARGM/2025/002309/MTR by Anusandhan National Research Foundation. Ram Baran Verma would like to thank Anusandhan National Research Foundation and National Board of Higher Mathematics for supporting him by grant ANRF/ARGM/2025/002357/MTR and 02011/36/2025/NBHM/RP/9466 respectively.

\bibliography{Reference}
\bibliographystyle{abbrv}
\end{document}